\documentclass[journal,twoside,web]{ieeecolor}
\usepackage{generic}
\usepackage[english]{babel}
\usepackage{cite}
\usepackage[euler]{textgreek}
\usepackage{amsmath,amssymb,amsfonts}
\usepackage{algorithmic}
\usepackage{graphicx}
\usepackage{textcomp}
\graphicspath{{images/}}
\usepackage[colorinlistoftodos]{todonotes}
\usepackage[T1]{fontenc}
\usepackage{lmodern}
\usepackage[scaled]{}
\usepackage{multicol}
\usepackage{multirow}
\usepackage{breqn}
\usepackage{float}
\usepackage{caption}
\usepackage{subcaption}
\usepackage{amstext}
\usepackage{xcolor}
\usepackage{bm}
\usepackage{trfsigns}
\usepackage{mathtools}
\usepackage[colorinlistoftodos]{todonotes}
\usepackage{authblk}
\usepackage{tcolorbox}
\usepackage{listings}
\usepackage{multicol}
\usepackage{multirow}
\usepackage{amstext}
\usepackage{mathtools}
\usepackage{lipsum}
\usepackage{graphicx}
\usepackage{epstopdf}
\usepackage{algorithm2e}
\usepackage{algorithmic}
\usepackage{subcaption}
\usepackage{placeins}
\usepackage{booktabs}
\usepackage{caption}
\usepackage{xcolor}
\usepackage{pseudocode}
\usepackage{filecontents}

\usepackage{epstopdf}
\usepackage{multicol}
\usepackage{multirow}
\usepackage{placeins}
\usepackage{booktabs}
\usepackage{mathrsfs} 
\usepackage{cleveref} 
\usepackage{accents}
\usepackage{empheq}
\usepackage{accents}
\usepackage{cite}
\usepackage{cleveref}

\usepackage{etoolbox}
\makeatletter
\patchcmd{\@makecaption}
{\scshape}
{}
{}
{}
\makeatletter
\patchcmd{\@makecaption}
{\\}
{ }
{}
{}
\makeatother

\newcommand{\norm}[1]{\left\lVert#1\right\rVert}
\newcommand{\dt}{\mathrm{d}t}
\newcommand{\dif}{\mathrm{d}}

\newcommand{\newsup}{\mathop{\smash{\mathrm{sup}}}}

\newtheorem{remark}{Remark}
\newtheorem{proposition}{Proposition}
\newtheorem{theorem}{Theorem}
\newtheorem{definition}{Definition}
\newtheorem{lemma}{Lemma}
\newtheorem{corollary}{Corollary}

\def\BibTeX{{\rm B\kern-.05em{\sc i\kern-.025em b}\kern-.08em
    T\kern-.1667em\lower.7ex\hbox{E}\kern-.125emX}}
\begin{document}
\title{Comparison of Non-deterministic Linear Systems by (\textgamma,\textdelta)-Similarity}
\author{Armin Pirastehzad, Arjan van der Schaft, \IEEEmembership{Fellow, IEEE}, and Bart Besselink, \IEEEmembership{Member, IEEE}
	\thanks{This publication is part of the project “Integration of Data-drIven and model-based enGIneering in fuTure industriAL Technology With value chaIn optimizatioN (DIGITAL TWIN)” (with project number P18-03 project 1 of the research program Perspectief which is (partly) financed by the Dutch Research Council (NWO).}
	\thanks{All authors are with the Bernoulli Institute for Mathematics, Computer Science and Artificial Intelligence, University of Groningen, 9700 AK Groningen, The Netherlands (email: a.pirastehzad@rug.nl; a.j.van.der.schaft@rug.nl; 	b.besselink@rug.nl).}
}

\maketitle

\begin{abstract}
We introduce (\textgamma,\textdelta)-similarity, a notion of system comparison that measures to what extent two stable linear dynamical systems behave similarly in an input-output sense. This behavioral similarity is characterized by measuring the sensitivity of the difference between the two output trajectories in terms of the external inputs to the two potentially non-deterministic systems. As such, (\textgamma,\textdelta)-similarity is a notion that characterizes \emph{approximation} of input-output behavior, whereas existing notions of simulation target equivalence. Next, as this approximation is specified in terms of the $\mathcal{L}_2$ signal norm, (\textgamma,\textdelta)-similarity allows for integration with existing methods for analysis and synthesis of control systems, in particular, robust control techniques. We characterize the notion of (\textgamma,\textdelta)-similarity as a linear matrix inequality feasibility problem and derive its interpretation in terms of transfer matrices. Our study on the compositional properties of (\textgamma,\textdelta)-similarity shows that the notion is preserved through series and feedback interconnections. This highlights its potential application in compositional reasoning, namely abstraction and modular synthesis of large-scale interconnected dynamical systems. We further illustrate our results in an electrical network example.    
\end{abstract}


\section{Introduction}
Modern engineering systems such as high-tech equipment, intelligent transportation systems, and smart grids have witnessed a massive growth in complexity. Due to lack of scalability in many existing analytic and synthetic methods, crucial problems such as specification verification (analysis), component replacement, and controller design (synthesis) have become increasingly challenging to address. These challenges motivate the adoption of system relations that compare the (dynamic) behavior of a complex system with that of a, typically simpler, one. Finding their applications in the analysis and synthesis of interconnected systems, system relations appear as the cornerstones of so-called modular approaches, \textit{e.g.,} \cite{EDA-053}, that allow for global system analysis (synthesis) on the basis of component-level analysis (synthesis). In this paper, we will develop a notion for \emph{approximate} comparison of system behavior, where the systems are possibly non-deterministic.

Various notions of system comparison have been conceived so far. Motivated by the notion of (bi)simulation in theoretical computer science \cite{milner1989communication}, the extension of (bi)simulation to continuous-time dynamical systems was explored in  \cite{lafferriere1997hybrid,pappas2000hierarchically,alur2000discrete,pappas2002consistent}
and formally presented in \cite{pappas2003bisimilar} using the framework of labeled transition systems. The notion of (bi)simulation was specialized to continuous-time systems in \cite{van2004equivalence}, where it was formulated as a modified disturbance decoupling problem and characterized algebraically. A notion of system equivalence, (bi)simulation finds its applications in compositional reasoning \cite{kerber2010compositional}, contract-based design \cite{besselink2019contracts,Shali2022series}, and controller synthesis \cite{tabuada2008controller,megawati2017abstraction}. A comprehensive treatment of the variations of (bi)simulation, such as alternating simulation, can be found in \cite{tabuada2009verification}.


As the notion of (bi)simulation and its variations fail to address the case where the external behavior of the systems are \emph{close} rather than identical, an approximate version of (bi)simulation was studied in  \cite{girard2007approximation,girard2007approximate}, where metric transition systems were used as a unified framework to treat both discrete and continuous systems. Approximate (bi)simulation relaxes the requirement on exact equivalence of external behaviors as it requires them to be sufficiently close, which is measured as a bound over the maximum distance between the external behaviors. For this reason, the notion of approximate (bi)simulation has become a keystone in a broad range of applications, namely compositional reasoning \cite{tazaki2008approximately}, controller synthesis \cite{girard2012controller,girard2009hierarchical,pola2019control} and fault detection \cite{pola2017approximate}. 
Several modifications to the original notion of approximate (bi)simulation have given rise to different variations.  An extensive treatment of modified versions of approximate (bi)simulation can be found in \cite{roehm2019model} and references therein. 

As far as continuous-time dynamical systems are concerned, the notion of approximate (bi)simulation and its variations characterize a bound on the distance between solution trajectories, using an $\mathcal{L}_\infty$ signal norm. Approximate (bi)simulation is however incompatible with many analytic and synthetic tools in control theory as they often employ $\mathcal{L}_2$ signal norm, \textit{e.g.,} in robust control \cite{van2000l2,zhou1998essentials} or dissipativity theory \cite{brogliato2007dissipative,scherer2000linear}. Indeed, the use of the $\mathcal{L}_2$ norm has many computational advantages. On the other hand, comparing trajectories on compact regions of the input-output space, approximate (bi)simulation evaluates the bound on the output error via a ``$\sup$-$\inf$'' optimization, \textit{i.e.,} two static games, which leads to additional \emph{complexities} and further computational costs.

The goal of this paper is to conceive a notion of system comparison, referred to as (\textgamma,\textdelta)-similarity, that 1) measures to what extent the input-output behaviors of continuous-time stable linear dynamical systems are similar in an $\mathcal{L}_2$ sense and 2) can be used for modular analysis (design) using compositional reasoning. The main contributions of the paper are as follows.

First, we introduce the notion of (\textgamma,\textdelta)-similarity which formulates behavioral similarity of two \emph{non-deterministic} stable linear dynamical systems as the sensitivity of the difference between the output trajectories to the external inputs to the two systems. These dynamical systems are non-deterministic in the sense that their solution trajectories may also depend on unknown disturbance signals. While (\textgamma,\textdelta)-similarity extends simulation in the sense that it characterizes input-output approximation rather than equivalence, it is different from approximate simulation. In contrast to approximate simulation where approximation metrics are constructed in an $\mathcal{L}_\infty$ sense, we develop a comparison metric that measures the sensitivity in terms of the $\mathcal{L}_2$ signal norm, which allows for the use of analytic and synthetic tools developed in control theory. This becomes crucial as one desires to enforce behavioral similarity through control design, \textit{i.e.,} one designs a controller to enforce solution trajectories to be as \emph{close} as possible. In such case, (\textgamma,\textdelta)-similarity gives one the benefit of using existing effective tools, \textit{e.g.,} $H_\infty$ synthesis \cite{green2012linear,scherer2000linear}. The notion of (\textgamma,\textdelta)-similarity proves to be suitable for system comparison as it admits reflexivity-like and transitivity-like properties.     

Second, we draw inspiration from $H_\infty$ control theory (see, \textit{e.g.,} \cite{green2012linear,trentelman2012control}) to show that (\textgamma,\textdelta)-similarity can be characterized through the existence of a solution to an algebraic Riccati equation (ARE). We then use ideas from dissipativity theory (see, \textit{e.g.,} \cite{van2000l2,willems1972dissipative,brogliato2007dissipative}) to formulate this ARE as a linear matrix inequality (LMI) feasibility problem, which is in terms of the system matrices and does not require the computation of system trajectories. In contrast to the notion of approximate simulation whose characterization is computationally demanding as it involves a sup-inf optimization problem, the notion of (\textgamma,\textdelta)-similarity is characterized as an LMI feasibility problem, for which many efficient computational tools are available. 

Third, we give an interpretation of the notion of (\textgamma,\textdelta)-similarity in terms of the transfer matrices of the systems. For deterministic dynamical systems, we show that the notion of (\textgamma,\textdelta)-similarity also measures the distance between the transfer matrices of the two systems in terms of the $H_\infty$ norm. 

Fourth, we study the extension of the notion of (\textgamma,\textdelta)-similarity to series and feedback interconnections of systems. We show that a series (feedback) interconnection is (\textgamma,\textdelta)-similar to another series (feedback) interconnection if each of its constituting components is (\textgamma,\textdelta)-similar to the corresponding component in the other interconnection (and a small-gain-type condition holds). This, as a result, indicates that the notion of (\textgamma,\textdelta)-similarity is preserved through series and feedback interconnections of systems, highlighting its applicability for compositional reasoning, which allows one to decompose interconnected system verification (design) into local component verification (design). 

The notion of (\textgamma,\textdelta)-similarity differs from notions of system comparison, \textit{e.g.,} \cite{rungger2016compositional,girard2009hierarchical,prabhakar2018simulations,prabhakar2016bisimulations,li2021input}, that employ (alternating) simulation relations to characterize the bound on the difference between the outputs in terms of the external inputs. Namely, such notions characterize a bound that also depends either on the system states, as in \cite{rungger2016compositional,girard2009hierarchical,prabhakar2018simulations,prabhakar2016bisimulations}, or system outputs, as in \cite{li2021input}. By contrast, (\textgamma,\textdelta)-similarity characterizes a bound solely in terms of the external signals, \textit{i.e.,} the bound does not depend on the system states. More importantly, these notions provide the bound in $\mathcal{L}_\infty$ signal norm, whereas (\textgamma,\textdelta)-similarity characterizes the bound in an $\mathcal{L}_2$ sense. While the former is not compatible with many effective analytic and synthetic tools in control theory (as they often employ $\mathcal{L}_2$ signal norm), the latter gives one the benefit of utilizing such tools. 

Despite comparing the external behavior of systems, the comparison metric developed for the notion of (\textgamma,\textdelta)-similarity also differs from the so-called gap metrics \cite{zames1980unstable,georgiou1997robustness}, where the external behavior of a nominal system and its perturbations are of interest. In contrast to these gap metrics, the notion of (\textgamma,\textdelta)-similarity addresses systems with state space equations of potentially different orders and disturbances of distinct dimensions.

The rest of this paper is organized as follows. In Section~\ref{DefinitionSection}, the notion of (\textgamma,\textdelta)-similarity is introduced, whereas its characterization is left for Section~\ref{CharacterizationSection}. In Section~\ref{DeterministicSystemsSection}, we study the notion of (\textgamma,\textdelta)-similarity for deterministic systems, whereas in Section~\ref{CompositionSection}, we employ (\textgamma,\textdelta)-similarity for compositional reasoning by investigating compositional properties. We further demonstrate our results in an illustrative example in Section~\ref{IllustrativeSection} and, finally, conclude the paper in Section~\ref{ConclusionSection}.

\paragraph*{Notation} For a matrix $M \in\mathbb{R}^{n\times n}$, we respectively denote its spectrum, largest eigenvalue, and largest singular value by $\operatorname{spec}(M)$, $\lambda_{\text{max}}(M)$, and $\sigma_{\text{max}}(M)$.  We define the operator $\operatorname{col}(\cdot,\cdot)$ such that $\operatorname{col}(x_1,x_2) = (x_1^T,x_2^T)^T$, for any vectors $x_1\in\mathbb{R}^{n_1}$ and $x_2\in\mathbb{R}^{n_2}$. Given a vector $x\in\mathbb{R}^n$ and a positive definite symmetric matrix $M$, $|x| = (x^Tx)^{1/2}$ and $|x|_{M} \  = (x^TMx)^{1/2}$ denote the Euclidean norm and the weighted Euclidean norm, respectively. We denote by $\mathcal{L}_2$ the space of measurable functions $u: [0,\infty) \rightarrow\mathbb{R}^n$ such that $\int_{0}^{\infty}|u(t)|^2\dt < \infty$. Accordingly, the $\mathcal{L}_2$ space is endowed with the norm $\norm{u} = (\int_{0}^{\infty}|u(t)|^2\dt)^{1/2}$. Given any $T>0$, we denote by $\mathcal{L}_2[0,T]$ the space of measurable functions $u: [0,\infty) \rightarrow\mathbb{R}^n$ such that $\int_{0}^{T}|u(t)|^2\dt < \infty$ and endow it with the norm $\norm{u}_{[0,T]} = (\int_{0}^{T}|u(t)|^2\dt)^{1/2}$. Eventually, given a positive definite matrix $M$, we define $\norm{u_1}_M = (\int_{0}^{\infty}|u_1(t)|_M^2\dt)^{1/2}$ and $\norm{u_2}_{[0,T],M} = (\int_{0}^{T}|u_2(t)|_M^2\dt)^{1/2}$, for any $u_1\in\mathcal{L}_2$ and $u_2\in\mathcal{L}_2[0,T]$. 

\section{(\textgamma,\textdelta)-Similarity}\label{DefinitionSection}
Consider continuous-time linear systems of the form 
\begin{equation}\label{System}
\bm{\Sigma}_i: \left\{\begin{aligned}
\dot{x}_i &= A_ix_i + B_iu_i + E_id_i,\\
y_i &= C_ix_i,
\end{aligned}
\right.
\end{equation}
where $x_i \in \mathbb{R}^{n_i}$, $u_i \in \mathbb{R}^{m_i}$, $d_i \in \mathbb{R}^{q_i}$, and $y_i \in \mathbb{R}^{p_i}$ represent system state, input, disturbance, and output, respectively. Throughout this paper, we assume that \eqref{System} is $0$-asymptotically stable, \textit{i.e.,} the system is asymptotically stable in the absence of inputs and disturbances, and therefore, the state matrix $A_i$ is Hurwitz. We denote by $y_i(t;t_0,x_{i,0},u_i,d_i)$ the output solution, at time $t$, of \eqref{System} for initial condition $x_i(t_0) = x_{i,0}$, input $u_i$, and the disturbance $d_i$. While the system state $x_i$ is considered as an internal variable, the input $u_i$ and the output $y_i$ are regarded as external variables through which $\bm{\Sigma}_i$ interacts with its environment. The presence of the disturbance $d_i$ leads to nondeterminism in the sense that trajectories do not solely depend on the initial condition $x_i(t_0)$ and the input $u_i$ (such non-deterministic systems may, for example, arise from abstraction, see for example \cite{pappas2000hierarchically}). 


We are interested in \emph{comparing} systems of the form \eqref{System} according to the following definition. 
\begin{definition}\label{Definition}
	Given $0$-asymptotically stable systems $\bm{\Sigma}_1$ and $\bm{\Sigma}_2$, for $\gamma,\delta>0$, the system $\bm{\Sigma}_2$ is said to be (\textgamma,\textdelta)-similar to system $\bm{\Sigma}_1$, denoted by $\bm{\Sigma}_1 \preccurlyeq_{\gamma,\delta} \bm{\Sigma}_2$, if there exist constants $\varepsilon,\eta,\mu>0$ such that for every input $u_1,u_2 \in \mathcal{L}_2$ and every disturbance $d_1 \in \mathcal{L}_2$, there exists a disturbance  $d_2 \in \mathcal{L}_2$ such that 
	\begin{equation}\label{GammaConformanceInequality}
	\begin{aligned}
	\norm{y_1-y_2}^2 \leq&\; \gamma\norm{u_1-u_2}^2 + (\delta-\varepsilon)\norm{\begin{bmatrix}
		u_1\\u_2
		\end{bmatrix}}^2\\
	& + (\mu-\varepsilon)\norm{d_1}^2 - \eta\norm{d_2}^2,
	\end{aligned}
	\end{equation}
	where $y_i(t) = y_i(t;0,0,u_i,d_i)$ for $i=1,2$.
\end{definition}

The notion of (\textgamma,\textdelta)-similarity in Definition~\ref{Definition} measures the similarity of the trajectories of $\bm{\Sigma}_1$ and $\bm{\Sigma}_2$ in terms of their input-output behavior. In \eqref{GammaConformanceInequality}, the (small) parameter $\varepsilon$ is merely present for technical reasons which will be explained later. 
When considering deterministic systems, \textit{i.e.,} in the absence of disturbances $d_i$, \eqref{GammaConformanceInequality} gives a bound on the output trajectories in terms of the input trajectories. In particular, the parameter $\gamma$ measures to what extent a dissimilarity in inputs leads to a deviation in outputs. On the other hand, the parameter $\delta$ accounts for how each input affects this deviation, which is essential as it also provides a bound on output dissimilarity in the case where $u_1 = u_2$. Then, for non-deterministic systems, constants $\mu$ and $\eta$ specify how nondeterminism contributes to output deviation. More specifically, when input signals $u_1$ and $u_2$ are given, Definition~\ref{Definition} states that for any trajectory of $\bm{\Sigma}_1$ (now determined by $d_1$), there exists a trajectory of $\bm{\Sigma}_2$ (through the choice of $d_2$) that approximates it according to \eqref{GammaConformanceInequality}. For this reason, the notion of (\textgamma,\textdelta)-similarity  provides us with a criterion which measures to what extent the input-output behavior of $\bm{\Sigma}_1$ is contained in that of $\bm{\Sigma}_2$. 
\begin{remark}
	Definition \ref{Definition} asks for the \emph{existence} of $\mu,\eta>0$ rather than giving them a priori. This allows the notion of (\textgamma,\textdelta)-similarity to properly measure how close the input-output behaviors of the systems are. In order to better appreciate this, consider the systems
	\begin{equation*}
	\begin{aligned}
	&\bm{\Sigma}_1: \ \begin{cases}
	\dot{x}_1 = -x_1 + u_1 + 2d_1,\\
	y_1 = x_1,
	\end{cases}
	\\
	&\bm{\Sigma}_2: \ \begin{cases}
	\dot{x}_2 = -x_2 + u_2 + d_2,\\
	y_2 = x_2.
	\end{cases}
	\end{aligned}
	\end{equation*}
	Taking $d_2 = 2d_1$, it is clear that there exists a $\gamma>0$ such that for any $\delta>0$, the system $\bm{\Sigma}_2$ is (\textgamma,\textdelta)-similar to system $\bm{\Sigma}_1$ with any choice $\varepsilon<\delta$ and $\mu > \eta + \varepsilon$. However, if the parameters $\mu$ and $\eta$ were fixed, \eqref{GammaConformanceInequality} would not hold for any $\delta>0$ unless the values of $\mu$ and $\eta$ were such that $\mu > \eta + \varepsilon$. Obviously, in this case, the notion would fail to properly measure the similarity of the input-output behavior of systems.
\end{remark}
\begin{remark}\label{ComparisonSimulation}
	The notion of (\textgamma,\textdelta)-similarity can be regarded as an extension of the notion of simulation, see \cite{van2004equivalence,lafferriere1997hybrid,pappas2000hierarchically,pappas2003bisimilar,pappas2002consistent}. Namely, whereas simulation is concerned with input-output \emph{equivalence} of non-deterministic systems, (\textgamma,\textdelta)-similarity characterizes \emph{approximation} of input-output behavior as it asks for neither identical inputs nor identical outputs. Still, (\textgamma,\textdelta)-similarity does \emph{not} follow from simulation relations that incorporate external inputs and treat disturbances as sources of non-determinism, \textit{e.g.,} see  \cite{van2004equivalence}. This is due to the fact that such relations are only concerned with the case where systems are subject to a joint input and does not provide any results on the output similarity when inputs are not the same (as in Definition~\ref{Definition}).
\end{remark}
\begin{remark}\label{Application}
	The notion of (\textgamma,\textdelta)-similarity finds its application in \emph{specification verification}. In particular, one may express specifications in terms of the input-output trajectories of a dynamical system (as in \cite{kerber2010compositional,besselink2019contracts,Shali2022series}) and utilize (\textgamma,\textdelta)-similarity to \emph{verify} whether the external behavior of a system satisfies such specifications in an approximate sense. 
	In this scheme, with $\bm{S}$ given as a dynamical system of the form \eqref{System} whose input-output trajectories capture the specifications, $\bm{\Sigma}\preccurlyeq_{\gamma,\delta}\bm{S}$ indicates to what extent the external behavior of $\bm{\Sigma}$ satisfies such specifications. 
	One may also use (\textgamma,\textdelta)-similarity for \emph{component replacement} in network systems. Utilizing (\textgamma,\textdelta)-similarity to measure the behavioral similarity of two components, one verifies whether such components admit sufficiently similar external behaviors such that one can be replaced by another without jeopardizing the performance of the network. This application is studied for electrical networks in Section~\ref{IllustrativeSection}. Finally, (\textgamma,\textdelta)-similarity also becomes relevant in \emph{system abstraction} as one can employ (\textgamma,\textdelta)-similarity to compare the behavioral similarity of a high-dimensional system with a low-dimensional system, \textit{i.e.,} one measures how precise a low-dimensional system \emph{abstracts} the external behavior of a high-dimensional system.  
\end{remark}

The following result further motivates (\textgamma,\textdelta)-similarity as notion of system comparison. 
\begin{proposition}\label{ReflexivityTransitivity}
	The notion of (\textgamma,\textdelta)-similarity admits the following properties:
	\begin{enumerate}
		\item There exists a $\gamma>0$ such that the system $\bm{\Sigma}$ is (\textgamma,\textdelta)-similar to itself for any $\delta>0$. 
		\item Suppose $\bm{\Sigma}_1 \preccurlyeq_{\gamma_{12},\delta_{12}} \bm{\Sigma}_2$ and $\bm{\Sigma}_2 \preccurlyeq_{\gamma_{23},\delta_{23}} \bm{\Sigma}_3$. Then, there exist $\gamma_{13},\delta_{13}>0$ such that $\bm{\Sigma}_1 \preccurlyeq_{\gamma_{13},\delta_{13}} \bm{\Sigma}_3$.  
	\end{enumerate}
\end{proposition}
\begin{proof}
	For the first part, consider the system $\bm{\Sigma}$ subject to the external input $u'$ and the disturbance $d'$. Let $u,u' \in \mathcal{L}_2$ and $d' \in \mathcal{L}_2$ be given and define $\tilde{x}_i := x - x'$ and $\tilde{y} := y - y'$. Taking $d = d'$, one obtains
	\begin{equation*}
	\begin{aligned}
	\dot{\tilde{x}} &= A\tilde{x} + B(u - u'),\\
	\tilde{y} &= C\tilde{x}.
	\end{aligned}
	\end{equation*}  
	As $A$ is Hurwitz and $u,u'\in\mathcal{L}_2$, there exists a $\gamma>0$ such that \cite[Section 4.1]{desoer2009feedback}
	\begin{equation*}
	\norm{y - y'}^2 \leq \gamma \norm{u - u'}^2,
	\end{equation*}
	where $y(t) = y(t;0,0,u,d)$. This can be written as \eqref{GammaConformanceInequality} for any given $\delta>0$, as long as one takes $\varepsilon \leq \delta$ and $\mu \geq \eta + \epsilon$.
	
	As for the second part, let $u_1,u_3\in \mathcal{L}_2$ and $d_1 \in \mathcal{L}_2$ be given. After choosing $u_2 = \frac{1}{2}(u_1+u_3)$, it follows from $\bm{\Sigma}_1 \preccurlyeq_{\gamma_{12},\delta_{12}} \bm{\Sigma}_2$ that there exist $\varepsilon_{12},\eta_{12},\mu_{12}>0$ and $d_2\in\mathcal{L}_2$ such that 
	\begin{equation*}
	\begin{aligned}
	\norm{y_1-y_2}^2 \leq &\;\frac{\gamma_{12}}{4}\norm{u_1-u_3}^2\\
	& + (\delta_{12}-\varepsilon_{12})\norm{\begin{bmatrix}
		u_1\\\frac{1}{2}(u_1+u_3)
		\end{bmatrix}}^2\\
	&+ (\mu_{12}-\varepsilon_{12})\norm{d_1}^2 - \eta_{12}\norm{d_2}^2,
	\end{aligned}
	\end{equation*}
	which, by triangular and Cauchy-Schwartz inequalities, leads to
	\begin{equation}\label{TransitivityInequalitySubstitution1}
	\begin{aligned}
	\norm{y_1-y_2}^2\leq &\;\frac{\gamma_{12}}{4}\norm{u_1-u_3}^2\\
	& + (\delta_{12}-\varepsilon_{12})\left(\frac{3}{2}\norm{u_1}^2 + \frac{1}{2}\norm{u_3}^2\right)\\
	&+ (\mu_{12}-\varepsilon_{12})\norm{d_1}^2 - \eta_{12}\norm{d_2}^2.
	\end{aligned}
	\end{equation} 
	For this $d_2$, and $u_2$ as chosen above, we have from $\bm{\Sigma}_2 \preccurlyeq_{\gamma_{23},\delta_{23}} \bm{\Sigma}_3$ that there exist $\varepsilon_{23},\eta_{23},\mu_{23}>0$ and $d_3 \in \mathcal{L}_2$ such that 
	\begin{equation}\label{TransitivityInequalitySubstitution2}
	\begin{aligned}
	\norm{y_2-y_3}^2 \leq &\;\frac{\gamma_{23}}{4}\norm{u_1-u_3}^2\\& + (\delta_{23}-\varepsilon_{23})\left(\frac{1}{2}\norm{u_1}^2 + \frac{3}{2}\norm{u_3}^2\right)\\
	&+ \mu_{23}\norm{d_2^*}^2 - \eta_{23}\norm{d_3}^,
	\end{aligned}
	\end{equation}
	Multiplying \eqref{TransitivityInequalitySubstitution1} and \eqref{TransitivityInequalitySubstitution2} by $\mu_{23}$ and $\eta_{12}$, respectively, and subsequently adding the resulting inequalities yields
	\begin{equation*}\label{Conformance13}
	\begin{aligned}
	\norm{y_1-y_3}^2 \leq & \; \frac{\mu_{23}\gamma_{12}+\eta_{12}\gamma_{23}}{2\min\{\mu_{23},\eta_{12}\}}\norm{u_1-u_3}^2\\ &+ \frac{3\left( \mu_{23}(\delta_{12}-\varepsilon_{12}) + \eta_{12}(\delta_{23}-\varepsilon_{23}) \right)}{\min\{\mu_{23},\eta_{12}\}}\norm{\begin{bmatrix}
		u_1\\u_3
		\end{bmatrix}}^2\\
	&+\frac{2\mu_{23}(\mu_{12}-\varepsilon_{12})}{\min\{\mu_{23},\eta_{12}\}}\norm{d_1}^2\\
	&- \frac{2\eta_{12}\eta_{23}}{\min\{\mu_{23},\eta_{12}\}}\norm{d_3}^2.
	\end{aligned}
	\end{equation*}
	It now follows from Definition~\ref{Definition} that $\bm{\Sigma}_1 \preccurlyeq_{\gamma_{13},\delta_{13}} \bm{\Sigma}_3$ with
	\begin{equation*}
	\gamma_{13} = \frac{\mu_{23}\gamma_{12}+\eta_{12}\gamma_{23}}{2\min\{\mu_{23},\eta_{12}\}}, \;\; \delta_{13} = \frac{3(\mu_{23}\delta_{12}+\eta_{12}\delta_{23})}{\min\{\mu_{23},\eta_{12}\}},
	\end{equation*}
	which proves the desired result. 
\end{proof}

The first  and the second statements of Proposition \ref{ReflexivityTransitivity} can be regarded as ``reflexivity-like'' and ``transitivity-like'' properties, respectively.
These properties, which are essential for comparison purposes, make (\textgamma,\textdelta)-similarity appealing for system comparison.  
\section{Characterization}\label{CharacterizationSection}
In this section, we will give an algebraic characterization of (\textgamma,\textdelta)-similarity, and by doing so, we eventually propose an easily verifiable criterion for the notion.

To do so, first, we provide an alternative formulation by defining the matrices
\begin{equation}\label{QR}
Q = \begin{bmatrix}
(\gamma+\delta)I&-\gamma I&0\\
-\gamma I&(\gamma+\delta)I&0\\0&0&\mu I
\end{bmatrix}, \;\;\;\; R = \begin{bmatrix}
I&0\\0&\eta I
\end{bmatrix},
\end{equation}
and rewriting \eqref{GammaConformanceInequality} as
\begin{equation}\label{AlternativeGammaConformanceInequality}
\begin{aligned}
\norm{\begin{bmatrix}
	y_1-y_2\\ d_2
	\end{bmatrix}}_{R}^2 \leq \norm{\begin{bmatrix}
	u_1\\u_2\\d_1
	\end{bmatrix}}_{Q}^2 - \varepsilon \norm{\begin{bmatrix}
	u_1\\u_2\\d_1
	\end{bmatrix}}^2,
\end{aligned}
\end{equation}
accordingly. Let $x:=\operatorname{col}(x_1,x_2)$, $w := \operatorname{col}(u_1,u_2,d_1)$, and 
$z = \operatorname{col}(y_1-y_2,d_2)$; motivated by the alternative representation \eqref{AlternativeGammaConformanceInequality},
we collect the dynamics of $\bm{\Sigma}_1$ and $\bm{\Sigma}_2$ to obtain the composite system 
\begin{equation}\label{CompositeSystem}
\begin{aligned}
\dot{x} &= Ax + Bd_2 + Ew,\\
z &= Cx + Dd_2, 
\end{aligned}
\end{equation}
where
\begin{equation}\label{CompositeSystemMatrices}
\begin{aligned}
A &= \begin{bmatrix}
A_1&0\\0&A_2
\end{bmatrix}, &&B = \begin{bmatrix}
0\\E_2
\end{bmatrix}, \; \; &&E = \begin{bmatrix}
B_1&0&E_1\\0&B_2&0
\end{bmatrix},\\
C &= \begin{bmatrix}
C_1&-C_2\\0&0
\end{bmatrix}, &&D = \begin{bmatrix}
0\\I
\end{bmatrix}.
\end{aligned}
\end{equation}
Needless to say, \eqref{CompositeSystem} is $0$-asymptotically stable as both $A_1$ and $A_2$ are Hurwitz by assumption. We denote by $z(t;t_0,x_0,d_2,w)$ the output solution (at time $t$) of \eqref{CompositeSystem} for initial condition $x(t_0) = x_0$ and input signals $d_2$ and $w$. Given a terminal time $T$,  we define the cost function
\begin{equation}\label{CostFunction}
J_T(\tau, x_s, d_2,w) = \int_{\tau}^{T}|z(t)|_R^2 \ - \ |w(t)|_Q^2 \ \dt,
\end{equation}
where $z(t) = z(t;\tau,x_s,d_2,w)$. Following this formulation, the following result alternatively states Definition \ref{Definition} in terms of the composite system \eqref{CompositeSystem} and cost function \eqref{CostFunction}.

\begin{proposition}\label{AlternativeDefinition}
	For $\gamma,\delta>0$, the system $\bm{\Sigma}_2$ is (\textgamma,\textdelta)-similar to system $\bm{\Sigma}_1$ if and only if there exist matrices $Q$ and $R$ with a structure as in \eqref{QR} and a constant $\varepsilon>0$ such that
	\begin{equation}\label{DissipativityRepresentation} \forall w \in \mathcal{L}_2, \exists d_2\in\mathcal{L}_2 : \; \lim\limits_{T\rightarrow\infty}J_T(0,0, d_2,w) \leq -\varepsilon \norm{w}^2.
	\end{equation} 
\end{proposition}
\begin{proof}
	To show necessity, suppose  $\bm{\Sigma}_1 \preccurlyeq_{\gamma,\delta} \bm{\Sigma}_2$, \textit{i.e.,} there exist $\varepsilon,\mu,\eta>0$ such that for every $u_1,u_2\in\mathcal{L}_2$ and every $d_1\in\mathcal{L}_2$, there exists a $d_2\in\mathcal{L}_2$ such that \eqref{AlternativeGammaConformanceInequality}, with $Q$ and $R$ defined as in \eqref{QR}, holds. Writing \eqref{AlternativeGammaConformanceInequality} in terms of the composite \eqref{CompositeSystem} and the cost function \eqref{CostFunction}, one immediately concludes \eqref{DissipativityRepresentation}.
	
	To show sufficiency, suppose there exist matrices $Q$ and $R$ with a structure as in \eqref{QR} and a constant $\varepsilon>0$ such that \eqref{DissipativityRepresentation} holds, which can be compactly written as 
	\begin{equation}\label{AuxilaryDefinition}
		\norm{z}_R^2 \leq \norm{w}_Q^2 - \varepsilon\norm{w}^2.
	\end{equation}
	Given the partioning $z = \operatorname{col}(y_1-y_2,d_2)$ and $w = \operatorname{col}(u_1,u_2,d_1)$, \eqref{AuxilaryDefinition} can be written as \eqref{GammaConformanceInequality}.
\end{proof}

From now on, we will mostly deal with this alternative characterization rather than the original one. 
\begin{remark}\label{ConnectiontoHinftyProblem}
	Through a few simple algebraic manipulations, we can readily rewrite \eqref{DissipativityRepresentation} as
	\begin{equation}\label{SupremumRepresentation}
	\forall w\in\mathcal{L}_2 \ \text{s.t.} \ w\neq 0, \exists d_2 \in \mathcal{L}_2: \; \frac{\norm{z}_{R}}{\norm{w}_{Q}} < 1.
	\end{equation}
	It is worth noting that the constant $\varepsilon$ insures the strictness of the inequality in \eqref{SupremumRepresentation}.
	The notion of (\textgamma,\textdelta)-similarity seeks a $d_2$ subject to which the composite system \eqref{CompositeSystem} establishes the performance criterion \eqref{SupremumRepresentation}. As a consequence, we may regard $d_2$ as the ``control input'' to \eqref{CompositeSystem} and view $w$ as the disturbance input, such that \eqref{SupremumRepresentation} resembles a full information sub-optimal $H_\infty$ control problem (see, \textit{e.g.,} \cite[Chapter 6]{green2012linear} and \cite[Chapter 13]{trentelman2012control}). However, our problem differs from a standard $H_\infty$ control formulation. Namely, in contrast to the full information sub-optimal $H_\infty$ control problem, which seeks the existence of a state feedback $d_2$ such that the closed-loop system establishes \eqref{SupremumRepresentation}, here, we seek a $d_2$ establishing \eqref{SupremumRepresentation} without any a priori assumptions on its structure. In other words, we seek $d_2$ in an open-loop fashion. In addition, we do not require $d_2$ to establish any stability properties. Finally, our problem seeks the existence of the constants $\mu,\eta>0$, whereas in the full information sub-optimal $H_\infty$ control problem, all constants are given beforehand. Nonetheless, we will draw inspiration from full-information sub-optimal $H_\infty$ control theory in our characterization of (\textgamma,\textdelta)-similarity. 
\end{remark}

In the remainder of this section, we will derive an algebraic necessary and sufficient condition for (\textgamma,\textdelta)-similarity by taking the following steps. First, we consider a weak version of (\textgamma,\textdelta)-similarity where \eqref{GammaConformanceInequality} only holds over a finite time interval. We characterize this weak version, which will be addressed as \emph{finite-time (\textgamma,\textdelta)-similarity}, and study how the notion of (\textgamma,\textdelta)-similarity relates to it. This paves the way towards an elegant characterization of (\textgamma,\textdelta)-similarity. Eventually, we exploit  dissipativity theory to propose an easily verifiable criterion for the notion of (\textgamma,\textdelta)-similarity. Let us start with the formal definition of finite-time (\textgamma,\textdelta)-similarity. 
\begin{definition}\label{FiniteTimeDefinition}
	Given $0$-asymptotically stable systems $\bm{\Sigma}_1$ and $\bm{\Sigma}_2$, for $\gamma,\delta>0$ and a terminal time $T>0$, the system $\bm{\Sigma}_2$ is said to be finite-time (\textgamma,\textdelta)-similar to system $\bm{\Sigma}_1$ over the time interval $[0,T]$, denoted by $\bm{\Sigma}_1 \preccurlyeq_{\gamma,\delta}^{T} \bm{\Sigma}_2$, if a constant $\varepsilon>0$ and positive definite matrices $Q$ and $R$, as defined in \eqref{QR}, exist such that:
	\begin{equation}\label{GammaMuConformanceInequality}
	\begin{aligned}
	\forall w\in\mathcal{L}_2[0,T], & \ \exists d_2\in\mathcal{L}_2[0,T]:\\
	 &J_T(0,0,d_2,w) \leq -\varepsilon \norm{w}_{[0,T]}^2.
	\end{aligned}
	\end{equation}
\end{definition}
Before we characterize the notion of finite-time (\textgamma,\textdelta)-similarity, we derive the following lemma.
\begin{lemma}\label{GammaMuConformanceLemma}
	Consider the system \eqref{CompositeSystem} and let positive definite matrices $Q$ and $R$ and terminal time $T>0$ be given. Then, \eqref{GammaMuConformanceInequality} holds if and only if the following differential Riccati equation admits a positive semi-definite solution over the interval $[0,T]$:
	\begin{equation}\label{DRE}
	\begin{aligned}
	-\dot{P} =&\; A^TP+PA + C^TRC + PEQ^{-1}E^TP\\
	 &- PB(D^TRD)^{-1}B^TP; \;\;\; P(T) =0.
	\end{aligned}
	\end{equation}
\end{lemma}
\vspace{2mm}
\begin{proof}
	The proof follows the same procedure as in the finite-horizon full-information $H_\infty$ control problem, see, \textit{e.g.,} \cite[Theorem 13.1]{trentelman2012control}. 
	
	Given a positive semi-definite solution of \eqref{DRE}, it follows from completion of squares that for
	$d_2(t) = -(D^TRD)^{-1}B^TP(t)x(t)$, we obtain
	$J_T(0,0,d_2,w) \leq 0$. Taking an approach similar to that of \cite[Theorem 13.1]{trentelman2012control}, it is straightforward to finally conclude \eqref{GammaMuConformanceInequality}. 
	
	Given \eqref{GammaMuConformanceInequality}, we adopt an approach similar to that of \cite[Theorem 13.1]{trentelman2012control} to conclude the existence of a positive semi-definite solution to \eqref{DRE}. First, we utilize Theorem 10.7 of \cite{trentelman2012control} to show that \eqref{DRE} admits a solution over some interval $[T_1,T]$. We then define 
	$$\mathscr{J}_T(\tau,x_s): = \newsup_{w \in \mathcal{L}_2[\tau,T]}\inf_{d_2 \in \mathcal{L}_2[\tau,T]}J_T(\tau,x_s,d_2,w),$$
	and follow the proof of \cite[Lemma 13.2]{trentelman2012control} to conclude that for all $\tau \in [T_1,T]$,
	\begin{equation}\label{DRECostFunctionRelation}
	\mathscr{J}_T(\tau,x_s) = x_s^TP(\tau)x_s,
	\end{equation}
	which indicates how the solution $P$ is related to $\mathscr{J}_T$. Accordingly, we exploit this to show that $P$ does indeed exist over $[0,T]$. For this purpose, once again, it follows from \cite[Theorem 13.1]{trentelman2012control} that for all 
	$0\leq \tau_1\leq \tau_2\leq T$ and all $x_s$,
	\begin{subequations}\label{Properties}
		\begin{equation}\label{CostFunctionMonotonicity}
		\mathscr{J}_T(\tau_1,x_{s}) \geq  \mathscr{J}_T(\tau_2,x_{s}),
		\end{equation}
		\text{which implies that for all $\tau \in [0,T]$,}
		\begin{equation}\label{CostFunctionBound}
		\mathscr{J}_T(0, x_s) \geq  \mathscr{J}_T(\tau,x_s).
		\end{equation}
	\end{subequations} 
	Following the proof of \cite[Theorem 13.1]{trentelman2012control}, it is straightforward to conclude the boundedness of $\mathscr{J}_T(0, x_s)$. Finally, combining this with \eqref{DRECostFunctionRelation} and \eqref{Properties}, we deduce that $P(\tau)$ is uniformly bounded over the interval $[T_1,T]$, which, by Theorem 10.7 of \cite{trentelman2012control}, implies the existence of the solution $P$ over the whole interval $[0,T]$. Such a solution is positive semi-definite, directly following from \eqref{DRECostFunctionRelation} and \eqref{CostFunctionMonotonicity}.
\end{proof}

After choosing matrices $Q$ and $R$ as in \eqref{QR}, it is clear that Lemma \ref{GammaMuConformanceLemma} characterizes finite-time (\textgamma,\textdelta)-similarity in terms of the differential Riccati equation (DRE). As our second step, we study how the notion of (\textgamma,\textdelta)-similarity is related to its finite-time counterpart. The following result provides this relation. For this purpose, we first state the following lemma whose proof is given in the Appendix.  
\begin{lemma}\label{GammaConformanceFiniteTimeGammaConformance}
	Consider the system \eqref{CompositeSystem} and let positive definite matrices $Q$ and $R$ be given. Suppose that \eqref{DissipativityRepresentation} holds. Then,  \eqref{GammaMuConformanceInequality} holds for any $T>0$.
\end{lemma}
With the choice of $Q$ and $R$ as in \eqref{QR}, the following corollary directly follows from Lemma~\ref{GammaConformanceFiniteTimeGammaConformance}.
\begin{corollary}\label{GammaConformanceFiniteTimeGammaConformanceCorollary}
	For $\gamma,\delta>0$, suppose that $\bm{\Sigma}_2$ is (\textgamma,\textdelta)-similar to $\bm{\Sigma}_1$. Then, $\bm{\Sigma}_2$ is finite-time (\textgamma,\textdelta)-similar to $\bm{\Sigma}_1$ over the time interval $[0,T]$, for any $T>0$.
\end{corollary}

As our final step towards characterizing the notion of (\textgamma,\textdelta)-similarity, we utilize Corollary~\ref{GammaConformanceFiniteTimeGammaConformanceCorollary} to extend the result of Lemma~\ref{GammaMuConformanceLemma} to the infinite-time case, which is indeed the notion of (\textgamma,\textdelta)-similarity. The proof of the following result can be found in the Appendix. 
\begin{theorem}\label{GammaConformanceLemma}
	For $\gamma,\delta>0$, $\bm{\Sigma}_2$ is (\textgamma,\textdelta)-similar to $\bm{\Sigma}_1$ if and only if there exist constants $\eta,\mu>0$ such that, for matrices $Q$ and $R$ as in \eqref{QR}, the algebraic Riccati equation 
	\begin{equation}\label{ARE}
	\begin{aligned}
	0 =&\; A^TP+PA+C^TRC\\
	&+PEQ^{-1}E^TP-PB(D^TRD)^{-1}B^TP
	\end{aligned}
	\end{equation}
	admits a positive semi-definite solution $P$ such that
	\begin{equation}\label{Stabilizing}
	\operatorname{spec}\Bigl(A - B(D^TRD)^{-1}B^TP + EQ^{-1}E^TP\Bigr) \subseteq\mathbb{C}_-.
	\end{equation}
\end{theorem}
\vspace*{1mm}

Theorem~\ref{GammaConformanceLemma} is crucial in the sense that it provides an algebraic characterization of the notion of (\textgamma,\textdelta)-similarity. Despite this, utilization of Theorem~\ref{GammaConformanceLemma} could be challenging, as finding $\eta,\mu>0$ such that \eqref{ARE} admits a solution requires solving the equation possibly for a large number of distinct values for $\eta$ and $\mu$. This is  neither desirable nor necessary, as we are only interested in the \emph{existence} of such $\eta$ and $\mu$.

Inspired by ideas from dissipativity theory, we derive conditions guaranteeing the existence of $\eta,\mu>0$ such that \eqref{ARE} admits a solution. With the aim of doing so, we propose an alternative characterization of (\textgamma,\textdelta)-similarity. Even though our problem does not a priori fit the full information $H_\infty$ control framework (see, Remark~\ref{ConnectiontoHinftyProblem}), Theorem~\ref{GammaConformanceLemma} leads to a Riccati equation that is very familiar to that framework. This, then, immediately guarantees that $d_2$ can also be obtained through state feedback, leading to the following lemma.  
\begin{lemma}\label{AlternativeGammaConformanceLemma}
	For $\gamma,\delta>0$, $\bm{\Sigma}_2$ is (\textgamma,\textdelta)-similar to $\bm{\Sigma}_1$ if and only if there exist constants $\varepsilon,\eta,\mu>0$ and matrix $F$ such that the composite system
	\begin{equation}\label{ClosedLoopCompositeSystemsLemma1}
	\begin{aligned}
	\dot{x} &= (A+BF)x + Ew, \; x(0)=0,\\
	z &= (C+DF)x,
	\end{aligned}
	\end{equation}
	with matrices \eqref{CompositeSystemMatrices} is  $0$-asymptotically stable and satisfies
	\begin{equation}\label{AlternativeDissipativityRepresentationLemma}
	\forall w\in\mathcal{L}_2: \; \norm{z}_R^2 - \norm{w}_Q^2 \leq -\varepsilon\norm{w}^2,
	\end{equation}
	for matrices $Q$ and $R$ given as in \eqref{QR}. Here, $z$ is the solution to \eqref{ClosedLoopCompositeSystemsLemma1}.
\end{lemma}
\begin{proof}
	\textit{If part}: Suppose there exist constants $\varepsilon,\eta,\mu>0$ and matrix $F$ such that \eqref{ClosedLoopCompositeSystemsLemma1} is $0$-asymptotically stable and \eqref{AlternativeDissipativityRepresentationLemma} is satisfied. As $A+BF$ is Hurwitz and $u,d_1\in\mathcal{L}_2$, one infers $Fx \in \mathcal{L}_2$. Defining $d_2 = Fx$, \eqref{ClosedLoopCompositeSystemsLemma1} takes the form \eqref{CompositeSystem}. Moreover, as $d_2 \in \mathcal{L}_2$, \eqref{AlternativeDissipativityRepresentationLemma} is equivalent to \eqref{DissipativityRepresentation}.
	As a result, by Definition~\ref{Definition}, system $\bm{\Sigma}_2$ is (\textgamma,\textdelta)-similar to system $\bm{\Sigma}_1$.
	
	\textit{Only if part}: Suppose $\bm{\Sigma}_2$ is (\textgamma,\textdelta)-similar to $\bm{\Sigma}_1$. It follows from Theorem~\ref{GammaConformanceLemma} that constants $\eta,\mu>0$ exist such that, for matrices $Q$ and $R$ as in \eqref{QR}, \eqref{ARE} admits a positive semi-definite solution satisfying \eqref{Stabilizing}. Taking $F = -(D^TRD)^{-1}B^TP$, it follows from the \textit{if} part of Theorem~\ref{GammaConformanceLemma} that subject to $d_2 = Fx$, \eqref{CompositeSystem} satisfies \eqref{TheoremAux5}. Nevertheless, subject to $d_2 = Fx$, the composite system \eqref{CompositeSystem} is of the form \eqref{ClosedLoopCompositeSystemsLemma1} with \eqref{TheoremAux5} equivalent to \eqref{AlternativeDissipativityRepresentationLemma}.
\end{proof}
\begin{remark}
	As pointed out in Remark~\ref{ConnectiontoHinftyProblem}, Definition~\ref{Definition} does not make any assumptions on the structure of $d_2$, a priori. However, Lemma~\ref{AlternativeGammaConformanceLemma} indicates that the existence of $d_2$, structured as a static state feedback, is implied by (\textgamma,\textdelta)-similarity. Therefore, closed-loop and open-loop choices of $d_2$ are equivalent. Under certain assumptions, a similar equivalence exists for the full-information sub-optimal $H_\infty$ control problem, see \cite[Theorem 3.1]{stoorvogel1992h}. It can be shown that our problem setup ensures satisfaction of these assumptions.
\end{remark}

Lemma~\ref{AlternativeGammaConformanceLemma} is consequential in the sense that it allows us to choose $d_2$ in a closed-loop fashion, \textit{i.e.,} as a static state feedback. The following result exploits this asset to give an interpretation of (\textgamma,\textdelta)-similarity in terms of dissipativity theory. 
\begin{lemma}\label{DissipativityLemma}
	For $\gamma,\delta>0$, $\bm{\Sigma}_2$ is (\textgamma,\textdelta)-similar to $\bm{\Sigma}_1$ if and only if there exist constants $\eta,\mu>0$ and a matrix $F$ such that the composite system 
	\begin{equation}\label{ClosedLoopCompositeSystemsLemma}
	\begin{aligned}
	\dot{x} &= (A+BF)x + Ew,\\
	z &= (C+DF)x,
	\end{aligned}
	\end{equation}
	is $0$-asymptotically stable and strictly dissipative with respect to the supply rate 
	\begin{equation}\label{SupplyFunction}
	s(w,z) = \begin{bmatrix}w\\z\end{bmatrix}^T\begin{bmatrix}
	Q&0\\0&-R
	\end{bmatrix}\begin{bmatrix}w\\z\end{bmatrix},
	\end{equation}
	\textit{i.e.,} there exists a function $V:\mathbb{R}^n\rightarrow[0,\infty)$ and an $\epsilon>0$ such that
	\begin{equation}\label{DissipationInequality}
	V(x(t_1)) \leq V(x(t_0)) + \int_{t_0}^{t_1}s(w(t),z(t))\dt -\epsilon\int_{t_0}^{t_1}|w(t)|^2\dt,
	\end{equation}
	for all $t_0\leq t_1$ and all signals $x$, $w$, and $z$ that satisfy \eqref{ClosedLoopCompositeSystemsLemma}.
\end{lemma}
\begin{proof}
	It follows from Lemma~\ref{AlternativeGammaConformanceLemma} that $\bm{\Sigma}_1\preccurlyeq_{\gamma,\delta}\bm{\Sigma}_2$ is equivalent to the existence of constants $\varepsilon,\eta,\mu>0$ and matrix $F$ such that \eqref{ClosedLoopCompositeSystemsLemma1} is $0$-asymptotically stable and satisfies \eqref{AlternativeDissipativityRepresentationLemma}. Nonetheless, it is straightforward to see that \eqref{AlternativeDissipativityRepresentationLemma} can be written as
	\begin{equation}\label{AlternativeDissipation}
	\forall w\in \mathcal{L}_2: \;\int_{0}^{\infty}\begin{bmatrix}
	w(t)\\z(t)
	\end{bmatrix}^T\begin{bmatrix}
	-Q+\varepsilon I&0\\0&R
	\end{bmatrix}\begin{bmatrix}
	w(t)\\z(t)
	\end{bmatrix}  \dt \leq 0.
	\end{equation}
	Let $G(s) = (C+DF)\bigl(sI-(A+BF)\bigr)^{-1}E$ be the transfer matrix of \eqref{ClosedLoopCompositeSystemsLemma1} and denote the Fourier transform of $w$ with $\hat{w}$. It follows immediately from Parseval's theorem (see, \textit{e.g.,} \cite[Section 4.2]{zhou1998essentials}) that $w \in \mathcal{L}_2$ if and only if $\hat{w} \in \mathcal{L}_2$. Moreover, Parseval's theorem implies that  \eqref{AlternativeDissipation} holds if and only if for all $\hat{w}\in\mathcal{L}_2$,
	\begin{equation*}
	\int_{0}^{\infty}\hat{w}^*(\im\omega)\begin{bmatrix}
	I\\G(\im\omega)
	\end{bmatrix}^*\begin{bmatrix}
	-Q+\varepsilon I&0\\0&R
	\end{bmatrix}\begin{bmatrix}
	I\\G(\im\omega)
	\end{bmatrix}\hat{w}(\im\omega) \ \mathrm{d}\omega\leq 0,
	\end{equation*}
	which, in turn, is equivalent to
	\begin{equation}\label{DissipativityCondition}
	\forall \omega: \; \begin{bmatrix}
	I\\G(\im\omega)
	\end{bmatrix}^*\begin{bmatrix}
	-Q&0\\0&R
	\end{bmatrix}\begin{bmatrix}
	I\\G(\im\omega)
	\end{bmatrix}\prec 0.
	\end{equation}
	Applying the Kalman-Yakubovich-Popov lemma  (see, \textit{e.g.,} \cite[Lemma 2.11]{scherer2000linear}), one concludes that \eqref{DissipativityCondition} is equivalent to the existence of $K = K^T$ such that
	\begin{equation}\label{DissipativityLMIAux}
	\begin{bmatrix}
	A_F^TK + KA_F + C_F^TRC_F&KE\\E^TK&-Q
	\end{bmatrix}\prec 0,
	\end{equation}
	where $A_F = A+BF$ and $C_F = C+ DF$.
	Since $R\succ 0$, it immediately follows from \eqref{DissipativityLMIAux} that 
	\begin{equation*}
	A_F^TK + KA_F\prec 0.
	\end{equation*}
	However, as $A_F$ is Hurwitz, this implies that $K \succ 0$. By a well-known result in dissipativity theory (see, \textit{e.g.,} \cite[Theorem 2.9]{scherer2000linear}), \eqref{DissipativityLMIAux} is equivalent to
	strict dissipativity of \eqref{ClosedLoopCompositeSystemsLemma} with respect to the supply rate \eqref{SupplyFunction} and the storage function $V(x) = x^TKx$.
\end{proof}

Given the dissipativity interpretation of the notion of (\textgamma,\textdelta)-similarity, we are now ready to derive the algebraic, easily verifiable, necessary and sufficient condition we desired.  
\begin{theorem}\label{GammaConformanceTheorem}
	For $\gamma,\delta>0$, $\bm{\Sigma}_2$ is (\textgamma,\textdelta)-similar to $\bm{\Sigma}_1$ if and only if there exist a positive definite matrix $X$, a matrix $\Pi$, and positive scalars $\tilde{\eta}$ and $\tilde{\mu}$ such that
	\begin{subequations}
		\begin{equation}\label{FinalLMI}
		\begin{aligned}
		&\begin{bmatrix}
		AX + B\Pi + (AX + B\Pi)^T &E&(CX+D\Pi)^T\\
		E^T&-\tilde{Q}(\gamma,\delta,\tilde{\mu})&0\\
		CX+D\Pi&0&-\tilde{R}(\tilde{\eta})
		\end{bmatrix}\\
		&\prec 0,
		\end{aligned}
		\end{equation}
		\text{where $A$, $B$, $E$, $C$, and $D$ are given as in \eqref{CompositeSystemMatrices} and}
		\begin{equation}\label{TildeQTildeR}
		\begin{aligned}
		\tilde{Q}(\gamma,\delta,\tilde{\mu}) &= \begin{bmatrix}
		(\gamma+\delta)I&-\gamma I&0\\-\gamma I&(\gamma+\delta)I&0\\0&0&\tilde{\mu} I
		\end{bmatrix},\\
		 \tilde{R}(\tilde{\eta}) &= \begin{bmatrix}
		I&0\\0&\tilde{\eta}I
		\end{bmatrix}.
		\end{aligned}
		\end{equation}
	\end{subequations}
\end{theorem}
\begin{proof}
	\textit{If part}: Suppose \eqref{FinalLMI} admits a solution. The congruence transformation
	\begin{equation}\label{CongruenceTransformation}
	\begin{bmatrix}
	X^{-1}&0&0\\0&I&0\\0&0&I
	\end{bmatrix},
	\end{equation}
	transforms \eqref{FinalLMI} into
	\begin{equation}\label{DissipativityLMIAux4}
	\begin{bmatrix}
	A_F^TX^{-1} + X^{-1}A_F&X^{-1}E&C_F^T\\E^TX^{-1}&-\tilde{Q}(\gamma,\delta,\tilde{\mu})&0\\
	C_F&0&-\tilde{R}(\tilde{\eta})
	\end{bmatrix}\prec 0,
	\end{equation}
	where $F:= \Pi X^{-1}$, $A_F = A+BF$, and $C_F = C+DF$. This suggests that $A_F$ is Hurwitz. Taking $Q = \tilde{Q}$ and $R = \tilde{R}^{-1}$, we take the Schur complement of \eqref{DissipativityLMIAux4} to obtain
	\begin{equation}\label{DissipativityLMI}
	\begin{bmatrix}
	A_F^TX^{-1} + X^{-1}A_F + C_F^TRC_F&X^{-1}E\\E^TX^{-1}&-Q
	\end{bmatrix}\prec 0.
	\end{equation}
	implying that \eqref{ClosedLoopCompositeSystemsLemma} is strictly dissipative with respect to the supply rate \eqref{SupplyFunction} and the storage function $V(x) = x^TX^{-1}x$ \cite[Theorem 2.9]{scherer2000linear}. Finally, it follows from Lemma~\ref{DissipativityLemma} that $\bm{\Sigma}_2$ is (\textgamma,\textdelta)-similar to $\bm{\Sigma}_1$.
	
	\textit{Only if part}: Suppose $\bm{\Sigma}_2$ is (\textgamma,\textdelta)-similar to $\bm{\Sigma}_1$. It follows from Lemma~\ref{DissipativityLemma} that \eqref{ClosedLoopCompositeSystemsLemma} is strictly dissipative with respect to the supply rate \eqref{SupplyFunction}, which implies the existence of a positive definite matrix $K$ such that \eqref{DissipativityLMIAux} holds.
	Taking $X = K^{-1}$, this is equivalent to \eqref{DissipativityLMI}. However, similar reasoning as in the \emph{if} part of the theorem shows that \eqref{DissipativityLMI} is equivalent to \eqref{FinalLMI}. 
\end{proof}

Theorem~\ref{GammaConformanceTheorem} characterizes the notion of (\textgamma,\textdelta)-similarity as a linear matrix inequality (LMI) feasibility problem, for which numerous computational techniques are available. We, thus, have equipped the notion with an easily verifiable necessary and sufficient condition.

An LMI feasibility problem characterizing the notion of (\textgamma,\textdelta)-similarity, \eqref{FinalLMI} involves an explicit computation of the parameters $\mu$ and $\eta$. Nevertheless, (\textgamma,\textdelta)-similarity merely asks for the \emph{existence} of such $\mu$ and $\eta$ rather than their values. This motivates us to modify the result of Theorem~\ref{GammaConformanceTheorem} such that it does not involve the computation of $\mu$ and $\eta$. The following result utilizes Finsler's lemma (see, \textit{e.g.,} \cite[Lemma 2]{oliveira2001stability}) to achieve this objective. 
\begin{theorem}\label{IndependentGammaConformanceTheorem}
	For $\gamma,\delta>0$, $\bm{\Sigma}_2$ is (\textgamma,\textdelta)-similar to $\bm{\Sigma}_1$ if and only if there exist a positive definite matrix $X$ and a matrix $\Pi$ such that
	\begin{equation}\label{IndependentLMI}
	\begin{bmatrix}
	XA^T+AX+\Pi^TB^T + B\Pi& \bar{E}&X\bar{C}^T\\
	\bar{E}^T&-\bar{Q}(\gamma,\delta)&0\\
	\bar{C}X&0&-I
	\end{bmatrix}\prec 0,
	\end{equation} 
	where $A$ and $B$ are given as in \eqref{CompositeSystemMatrices} and
	\begin{equation}\label{BarMatrices}
	\begin{aligned}
	\bar{E} &= \begin{bmatrix}
	B_1&0\\0&B_2
	\end{bmatrix},
	\bar{C} = \begin{bmatrix}
	C_1&-C_2
	\end{bmatrix},\\
	\bar{Q}(\gamma,\delta)&=\begin{bmatrix}
	(\gamma+\delta)I&-\gamma I\\-\gamma I&(\gamma+\delta)I
	\end{bmatrix}.
	\end{aligned}
	\end{equation}
\end{theorem}
\vspace*{2mm}
\begin{proof}
	As a result of Theorem~\ref{GammaConformanceTheorem}, $\bm{\Sigma}_1\preccurlyeq_{\gamma,\delta}\bm{\Sigma}_2$ is equivalent to the existence of $X\succ0$, $\Pi$, and $\tilde{\mu},\tilde{\eta}>0$ such that \eqref{FinalLMI} holds, where matrices $\tilde{Q}(\gamma,\delta,\tilde{\mu})$ and $\tilde{R}(\tilde{\eta})$ can be written as
	\begin{equation*}
	\begin{aligned}
	\tilde{Q}(\gamma,\delta,\tilde{\mu}) &= \begin{bmatrix}
	\bar{Q}(\gamma,\delta)&0\\0&0\end{bmatrix} + \tilde{\mu}\begin{bmatrix}
	0\\I
	\end{bmatrix}\begin{bmatrix}
	0&I
	\end{bmatrix},\\
	\tilde{R}(\tilde{\eta}) &= \begin{bmatrix}
	I&0\\0&0
	\end{bmatrix} + \tilde{\eta}\begin{bmatrix}
	0\\I
	\end{bmatrix}\begin{bmatrix}
	0&I
	\end{bmatrix}.
	\end{aligned}
	\end{equation*}
	However, applying Finsler's lemma twice, one concludes that \eqref{FinalLMI} is equivalent to \eqref{IndependentLMI} \cite[Lemma 2]{oliveira2001stability}. 
\end{proof}
\begin{remark}
	As for \eqref{IndependentLMI}, let us take $F = \Pi X^{-1}$. As a consequence, through a similar procedure as in obtaining \eqref{DissipativityLMI}, \eqref{IndependentLMI} is seen to be equivalent to
	\begin{equation*}
	\begin{bmatrix}
	(A+BF)^TX^{-1}+X^{-1}(A+BF)+\bar{C}^T\bar{C}&X^{-1}\bar{E}\\
	\bar{E}^TX^{-1}&-\bar{Q}(\gamma,\delta)
	\end{bmatrix}\prec 0,
	\end{equation*} 
	indicating that the composite system
	\begin{equation}\label{NominalSystem}
	\begin{aligned}
	\dot{x} &= (A+BF)x + \bar{E}\bar{w},\\
	\bar{z} &= \bar{C}x,
	\end{aligned}
	\end{equation} 
	where $\bar{w} = \operatorname{col}(u_1,u_2)$, is strictly dissipative with the supply function
	\begin{equation}\label{NominalSupplyFunction}
	\bar{s}(\bar{w},\bar{z}) = \begin{bmatrix}
	\bar{w}\\\bar{z}
	\end{bmatrix}^T\begin{bmatrix}
	\bar{Q}(\gamma,\delta)&0\\0&-I
	\end{bmatrix}\begin{bmatrix}\bar{w}\\\bar{z}
	\end{bmatrix},
	\end{equation}
	and the storage function $V(x) = x^TX^{-1}x$ \cite[Theorem 2.9]{scherer2000linear}. To summarize, it follows from Theorem~\ref{IndependentGammaConformanceTheorem} that the notion of (\textgamma,\textdelta)-similarity is fully characterized in terms of the strict dissipativity of the composite system \eqref{NominalSystem} with respect to the supply rate \eqref{NominalSupplyFunction}.
\end{remark}
\begin{remark}\label{NonzeroInitialCondition}
	The notion of (\textgamma,\textdelta)-similarity merely involves zero initial conditions, as it measures the similarity of forced responses. One may exploit Lemma~\ref{DissipativityLemma} to address the case of non-zero initial conditions. It follows from Lemma~\ref{DissipativityLemma} that  $\bm{\Sigma}_1\preccurlyeq_{\gamma,\delta}\bm{\Sigma}_2$ is equivalent to the existence of constants $\varepsilon,\eta,\mu>0$, a matrix $F$, and a positive definite matrix $K$ such that \eqref{ClosedLoopCompositeSystemsLemma} is $0$-asymptotically stable and satisfies 
	\begin{equation*}
	\begin{aligned}
	|x(t)|_K^2 - |x(0)|_K^2 \ \leq \int_{0}^{t}s(w(\tau),z(\tau))\mathrm{d}\tau-\varepsilon\int_{0}^{t}|w(t)|^2 \ \mathrm{d}\tau,
	\end{aligned}
	\end{equation*}
	where $s$ is given as in \eqref{SupplyFunction}. Because $A+BF$ is Hurwitz and $w\in\mathcal{L}_2$, as $t\rightarrow\infty$, the integrals remain well-defined and the state $x$ approaches the origin, which implies
	\begin{equation*}
	\forall w\in\mathcal{L}_2:\;\norm{z}_R^2 \leq \norm{w}_Q^2 - \varepsilon\norm{w}^2 + |x(0)|_K^2.
	\end{equation*}
	This, in turn, indicates that there exist constants $\varepsilon,\eta,\mu>0$ and a positive definite matrix $K$ such that for every input $u_1,u_2\in\mathcal{L}_2$ and every disturbance $d_1\in\mathcal{L}_2$, there exists a disturbance $d_2\in\mathcal{L}_2$ such that for any $x_{1,0}\in\mathbb{R}^{n_1}$ and $x_{2,0}\in\mathbb{R}^{n_2}$,
	\begin{equation}\label{NonZeroInitialConditionGammaConformanceInequality}
	\begin{aligned}
	\norm{y_1-y_2}_{\mathcal{L}_2}^2 \leq &\;\gamma\norm{u_1-u_2}_{\mathcal{L}_2}^2 + (\delta-\varepsilon)\norm{\begin{bmatrix}
		u_1\\u_2
		\end{bmatrix}}_{\mathcal{L}_2}^2\\
	&+ (\mu-\varepsilon)\norm{d_1}_{\mathcal{L}_2}^2- \eta\norm{d_2}_{\mathcal{L}_2}^2\\
	& + \left\vert\begin{bmatrix}
		x_{1,0}\\x_{2,0}
	\end{bmatrix}\right\vert_K^2,
	\end{aligned}
	\end{equation}
	where $y_i = y_i(t;0,x_{i,0},u_i,d_i)$. This is pivotal as it indicates that the notion of (\textgamma,\textdelta)-similarity also gives a measure on the similarity of general responses. As a result, while comparing the input-output behavior of two systems, we may confine ourselves to the forced response rather than the general response.	
\end{remark} 
\section{Deterministic Systems}\label{DeterministicSystemsSection}
In this section, we consider deterministic systems, \textit{i.e.,} systems of the form \eqref{System} with $d_i = 0$. In this case, \eqref{System} reduces to 
\begin{equation}\label{DeterministicPlant}
\bm{\Sigma}_i: \left\{\begin{aligned}
\dot{x}_i &= A_ix_i + B_iu_i,\\
y_i &= C_ix_i,
\end{aligned}
\right.
\end{equation}
where we maintain the standing assumption that $A_i$ is Hurwitz.
Adopting a notation in line with \eqref{System}, we denote by $y_i(t; t_0,x_{i,0},u_i)$ the output solution of \eqref{DeterministicPlant} at time $t$, for initial condition $x_{i,0}$ and input $u_i$. Considering $\bm{\Sigma}_i$ given by \eqref{DeterministicPlant}, for $\gamma,\delta>0$, $\bm{\Sigma}_2$ is (\textgamma,\textdelta)-similar to $\bm{\Sigma}_1$ if there exist a constant $\varepsilon>0$ such that
\begin{equation}\label{DeterministicGammaConformanceInequality}
\forall u_1,u_2\in\mathcal{L}_2: \; \norm{y_1-y_2}^2 \leq \gamma\norm{u_1-u_2}^2 + (\delta-\varepsilon)\norm{\begin{bmatrix}
	u_1\\u_2
	\end{bmatrix}}^2,
\end{equation}
where $y_i(t) = y_i(t;0,0,u_i)$. From \eqref{DeterministicGammaConformanceInequality}, one immediately cocnludes that for deterministic systems, the notion of (\textgamma,\textdelta)-similarity is symmetric in the sense that $\bm{\Sigma}_1 \preccurlyeq_{\gamma,\delta} \bm{\Sigma}_2$ and $\bm{\Sigma}_2 \preccurlyeq_{\gamma,\delta} \bm{\Sigma}_1$ are equivalent. For deterministic systems, Theorem \ref{GammaConformanceTheorem} readily specializes to the following result.
\begin{corollary}
	Consider systems $\bm{\Sigma}_1$ and $\bm{\Sigma}_2$ given by \eqref{DeterministicPlant}. For $\gamma,\delta>0$, $\bm{\Sigma}_2$ is (\textgamma,\textdelta)-similar to $\bm{\Sigma}_1$ (or equivalently, $\bm{\Sigma}_1$ is (\textgamma,\textdelta)-similar to $\bm{\Sigma}_2$) if and only if there exist a positive definite matrix $X$ such that
	\begin{equation}\label{DeterministicLMI}
	\begin{bmatrix}
	A^TX + XA + \bar{C}^T\bar{C}&X\bar{E}\\\bar{E}^TX&-\bar{Q}(\gamma,\delta)
	\end{bmatrix}\prec 0,
	\end{equation}
	where $A$ in given as in \eqref{CompositeSystemMatrices} while $\bar{C}$, $\bar{E}$, and $\bar{Q}(\gamma,\delta)$ are given as in \eqref{BarMatrices}.
\end{corollary}
As one would expect, \eqref{DeterministicLMI} is seen to be easily obtained from \eqref{FinalLMI} as one eliminates the rows and columns corresponding to $d_1$ and $d_2$. Nevertheless, a dissipativity formulation of \eqref{DeterministicGammaConformanceInequality} also gives \eqref{DeterministicLMI} (see, \textit{e.g.,} \cite[Proposition 3.9]{scherer2000linear}). 

We regard $G_i(s) = C_i(sI-A_i)^{-1}B_i$ as the transfer matrix of \eqref{DeterministicPlant}; as $A_i$ is Hurwitz, we can define the $H_\infty$ norm of $G_i$ as $$\norm{G_i}_\infty = \newsup_{\omega}\sigma_{\text{max}}(G_i(\im\omega)).$$ The following result, exploring the implications (\textgamma,\textdelta)-similarity has on the transfer matrices $G_1$ and $G_2$, immediately follows from \eqref{DeterministicGammaConformanceInequality}.
\begin{proposition}\label{IndividualDetereminism}
	For $\gamma,\delta>0$, if $\bm{\Sigma}_2$ is (\textgamma,\textdelta)-similar to $\bm{\Sigma}_1$ (or equivalently, $\bm{\Sigma}_1$ is (\textgamma,\textdelta)-similar to $\bm{\Sigma}_2$), then the following statements hold:
	\begin{enumerate}
		\item[a)] $\norm{G_i}_{\infty} < (\gamma+\delta)^{\frac{1}{2}}, \;\;\; i=1,2;$ 
		\item[b)]
		$\norm{G_1-G_2}_{\infty} < (2\delta)^{\frac{1}{2}}.$
	\end{enumerate}
\end{proposition}
\begin{proof}
	Following $\bm{\Sigma}_1 \preccurlyeq_{\gamma,\delta} \bm{\Sigma}_2$, there exists $\varepsilon>0$ such that \eqref{DeterministicGammaConformanceInequality} holds. Clearly, the choice of $u_2=0$ results in $y_2 = 0$. Accordingly, \eqref{DeterministicGammaConformanceInequality} implies 
	$$\forall u_1 \in\mathcal{L}_2: \; \norm{y_1}^2 \leq (\gamma + \delta - \varepsilon)\norm{u_1}^2,$$
	indicating that $\norm{G_1}_\infty<(\gamma+\delta)^{\frac{1}{2}}$ \cite[Theorem 4.3]{zhou1998essentials}. A similar reasoning indicates the same bound on $\norm{G_2}_\infty$. On the other hand, choosing $u: = u_1 = u_2 \in \mathcal{L}_2$, it follows from \eqref{DeterministicGammaConformanceInequality} that
	$$\norm{y_1-y_2}^2 \leq  2(\delta-\varepsilon)\norm{u}^2,$$
	which implies (b).
\end{proof}
Proposition~\ref{IndividualDetereminism} gives another interpretation of the role that $\gamma$ and $\delta$ play. In particular, the parameter $\delta$ measures the similarity of $G_1$ and $G_2$ in the sense of the $H_\infty$ norm. Consequently, $\delta$ suggests how precise $\bm{\Sigma}_2$ replicates the input-output behavior of $\bm{\Sigma}_1$. In addition, the parameter $\gamma$, together with $\delta$, indicates how the external input $u_i$ affects the output $y_i$. 

\section{Compositional Reasoning}\label{CompositionSection}
As many systems can be viewed as an interconnection of subsystems, it is often desirable to study such systems in terms of their constituting subsystems. This motivates one to explore how (\textgamma,\textdelta)-similarity on a `system' level may be concluded from (\textgamma,\textdelta)-similarity on a 	`component' level. Our goal is therefore to extend the notion from individual systems to the interconnections they constitute.

In preparation for this compositional reasoning, we address the following property which is an immediate consequence of (\textgamma,\textdelta)-similarity. The proof of the following Proposition is given in the Appendix.
\begin{proposition}\label{FiniteGainProperty}
	Suppose $\bm{\Sigma}_1 \preccurlyeq_{\gamma,\delta} \bm{\Sigma}_2$ for some $\gamma,\delta>0$. Then, there exist $\ell,k>0$ such that
	\begin{equation}\label{FiniteGain}
	\forall u_1,u_2,d_1\in\mathcal{L}_2, \exists d_2 \in \mathcal{L}_2: \; \norm{\begin{bmatrix}
		y_1\\y_2
		\end{bmatrix}}^2 \leq \ell \norm{\begin{bmatrix}
		u_1\\u_2
		\end{bmatrix}}^2 + k\norm{d_1}^2,
	\end{equation}
	where $y_i(t) = y_i(t;0,0,u_i,d_i)$.
\end{proposition}    

We commence our study with characterizing the (\textgamma,\textdelta)-similarity of systems structured as series interconnections, defined as follows.
\begin{definition}
	The system $\bm{\Sigma}_1$ is said to be in a series interconnection with the system $\bm{\Sigma}_2$, denoted by $\bm{\Sigma}_1 \parallel_{\text{s}} \bm{\Sigma}_2$, whenever $u_2 = y_1$, see Figure~\ref{Series}. Taking $x_s := \operatorname{col}(x_1,x_2)$, $u_s := u_1$, $d_s := \operatorname{col}(d_1,d_2)$, and $y_s := y_2$, the system $\bm{\Sigma}_1 \parallel_{\text{s}} \bm{\Sigma}_2$ is seen to be governed by the dynamics
	\begin{equation*}
	\begin{aligned}
	\dot{x}_s &= A_sx_s + B_su_s + E_sd_s,\\
	y_s &= C_sx_s,
	\end{aligned}
	\end{equation*}
	where
	\begin{equation*}
	\begin{aligned}
	A_s &= \begin{bmatrix}
	A_1&0\\
	B_2C_1&A_2
	\end{bmatrix}, \; B_s = \begin{bmatrix}
	B_1\\0
	\end{bmatrix}, \; E_s = \begin{bmatrix}
	E_1&0\\0&E_2
	\end{bmatrix},\\
	C_s &= \begin{bmatrix}
	0&C_2
	\end{bmatrix}.
	\end{aligned}
	\end{equation*}
\end{definition}
\vspace*{2mm}
It immediately follows from the $0$-asymptotic stability of $\bm{\Sigma}_1$ and $\bm{\Sigma}_2$ that $\bm{\Sigma}_1 \parallel_s \bm{\Sigma}_2$ is also $0$-asymptotically stable. 
\begin{figure}
	\centering
	\includegraphics[width=0.3\textwidth]{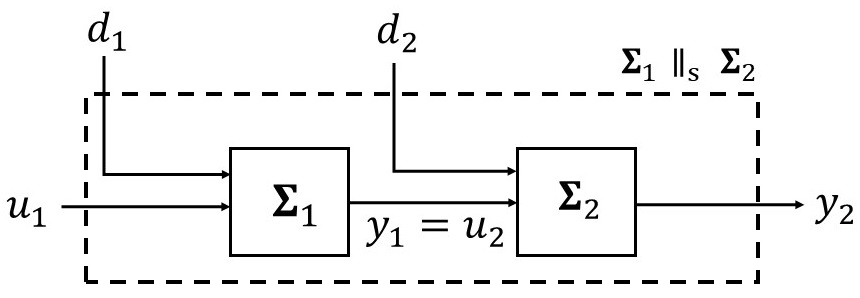}
	\caption{The series interconnection of systems $\bm{\Sigma}_1$ and $\bm{\Sigma}_2$.}
	\label{Series}
\end{figure} 

The following proposition explores the extension of (\textgamma,\textdelta)-similarity to series interconnection. The proof of the proposition can be found in the Appendix. 
\begin{proposition}\label{SeriesGammaConformance}
	Suppose $\bm{\Sigma}_1 \preccurlyeq_{\gamma_{1},\delta_{1}} \bm{\Sigma}_{1}'$ and $\bm{\Sigma}_2 \preccurlyeq_{\gamma_{2},\delta_{2}} \bm{\Sigma}_{2}'$. There exist $\gamma_s,\delta_s>0$ such that
	\begin{equation*}
	\bm{\Sigma}_1 \parallel_{\text{s}} \bm{\Sigma}_2 \preccurlyeq_{\gamma_s,\delta_s} \bm{\Sigma}_{1}' \parallel_{\text{s}} \bm{\Sigma}_{2}'.
	\end{equation*}
\end{proposition} 
\vspace*{1mm}
Proposition~\ref{SeriesGammaConformance} allows for \emph{modular} system comparison which is crucial for specification verification as it enables one to conduct global specification verification on the basis of local specifications. 

We are now interested in characterizing the (\textgamma,\textdelta)-similarity of a more complicated structure, \textit{i.e.,} feedback interconnection. 
\begin{definition}
	The system $\bm{\Sigma}_1$ is said to be in a feedback interconnection with the system $\bm{\Sigma}_2$, denoted by $\bm{\Sigma}_1 \parallel_{\text{f}} \bm{\Sigma}_2$, whenever $u_1 = e_1 - y_2$ and $u_2 = e_2 + y_1$, where $e_1$ and $e_2$ are external signals, see Figure~\ref{Feedback}. Taking $x_f := \operatorname{col}(x_1,x_2)$, $u_f := \operatorname{col}(e_1,e_2)$, $d_f := \operatorname{col}(d_1,d_2)$, and $y_f := \operatorname{col}(y_1,y_2)$, the system $\bm{\Sigma}_1 \parallel_{\text{f}} \bm{\Sigma}_2$ is seen to be governed by the dynamics
	\begin{equation*}
	\begin{aligned}
	\dot{x}_f &= A_fx_f + B_fu_f + E_fd_f,\\
	y_f &= C_fx_f,
	\end{aligned}
	\end{equation*}
	where
	\begin{equation*}
	\begin{aligned}
	A_f &= \begin{bmatrix}
	A_1&-B_1C_2\\
	B_2C_1&A_2
	\end{bmatrix}, \; B_f = \begin{bmatrix}
	B_1&0\\0&B_2
	\end{bmatrix}, \\
	 E_f &= \begin{bmatrix}
	E_1&0\\0&E_2
	\end{bmatrix},\qquad\quad\;\;
	C_f = \begin{bmatrix}
	C_1&0\\
	0&C_2
	\end{bmatrix}.
	\end{aligned}
	\end{equation*}
\end{definition}
\vspace*{2mm}

As one would expect, $\bm{\Sigma}_1 \parallel_{\text{f}} \bm{\Sigma}_2$ is not necessarily $0$-asymptotically stable. As the notion of (\textgamma,\textdelta)-similarity asks for $0$-asymptotically stable systems, we may only consider the case where $\bm{\Sigma}_1 \parallel_{\text{f}} \bm{\Sigma}_2$ is $0$-asymptotically stable, which in turn guarantees the internal stability of the interconnection. As a result, from now on, we assume $\bm{\Sigma}_1$ and $\bm{\Sigma}_2$ are such that $A_f$ is Hurwitz. 
\begin{figure}
	\centering
	\includegraphics[width=0.3\textwidth]{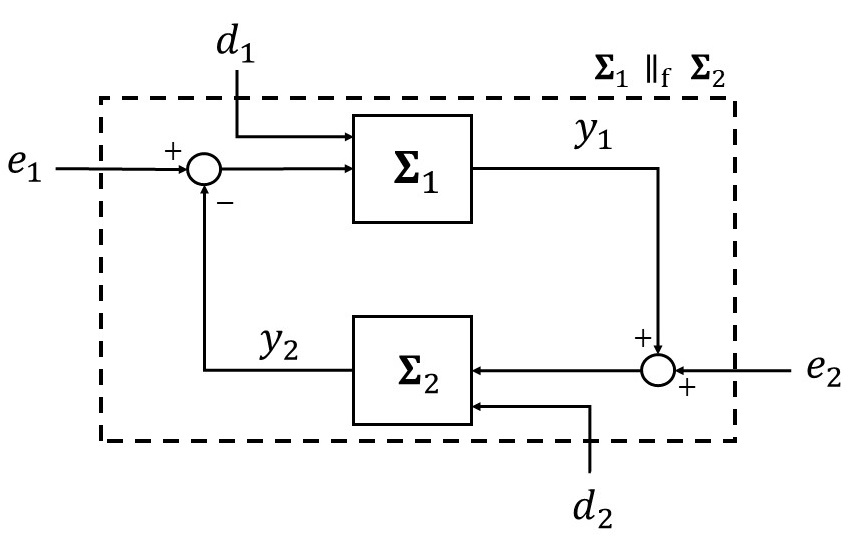}
	\caption{The feedback interconnection of systems $\bm{\Sigma}_1$ and $\bm{\Sigma}_2$.}
	\label{Feedback}
\end{figure} 

In contrast to series interconnection, the (\textgamma,\textdelta)-similarity of feedback interconnections does not automatically follow the (\textgamma,\textdelta)-similarity of the constituting subsystems. The following result, whose proof can be found in the Appendix, characterizes the condition under which (\textgamma,\textdelta)-similarity is preserved through feedback interconnection.  
\begin{proposition}\label{FeedbackGammaConformance}
	Suppose $\bm{\Sigma}_1 \preccurlyeq_{\gamma_{1},\delta_{1}} \bm{\Sigma}_{1}'$ and $\bm{\Sigma}_2 \preccurlyeq_{\gamma_{2},\delta_{2}} \bm{\Sigma}_{2}'$ such that for some $\epsilon>0$, 
	\begin{equation}\label{SmallGain}
	\gamma_{1}\gamma_{2}<\frac{1}{(1+\epsilon)^2}, \;\; \ell_{1}\ell_{2}<\frac{1}{(1+\epsilon)^2},
	\end{equation}
	with $\ell_{1},\ell_{2}$ given in \eqref{FiniteGain}, and $\bm{\Sigma}_1 \parallel_{\text{f}} \bm{\Sigma}_2$ and $\bm{\Sigma}_{1}' \parallel_{\text{f}} \bm{\Sigma}_{2}'$ are $0$-asymptotically stable. Then, there exist $\gamma_f,\delta_f>0$ such that 
	\begin{equation*}
	\bm{\Sigma}_1 \parallel_{\text{f}} \bm{\Sigma}_2 \preccurlyeq_{\gamma_f,\delta_f} \bm{\Sigma}_{1}' \parallel_{\text{f}} \bm{\Sigma}_{2}'.
	\end{equation*} 
\end{proposition} 
\vspace*{1mm}
\begin{remark}
	The parameter $\epsilon$ provokes a trade-off between the permissible domain of the gains $\gamma_1$, $\gamma_2$, $\ell_1$, and $\ell_2$ and the magnitude of the parameters $\gamma_f$ and $\delta_f$ as follows from the proof of Proposition~\ref{FeedbackGammaConformance}. For larger choices of $\epsilon$,  Proposition~\ref{FeedbackGammaConformance} is applicable to a smaller class of systems, \textit{i.e.,} systems with smaller $\gamma_1$, $\gamma_2$, $\ell_1$, and $\ell_2$. Nevertheless, for such choices, the resulting $\gamma_f$ and $\delta_f$ are smaller. In contrast, smaller choices of $\epsilon$ lead to a broader applicability of Proposition~\ref{FeedbackGammaConformance} with larger $\gamma_f$ and $\delta_f$. 
\end{remark}
\begin{remark}
	As (\textgamma,\textdelta)-similarity asks for $0$-asymptotically stable systems, we restricted ourselves to $0$-asymptotically stable feedback interconnections. Nevertheless, given the conditions of Proposition~\ref{FeedbackGammaConformance}, the bound on the output deviation holds regardless of $\bm{\Sigma}_1 \parallel_{\text{f}} \bm{\Sigma}_2$ and $\bm{\Sigma}_1' \parallel_{\text{f}} \bm{\Sigma}_2'$ being $0$-asymptotically stable. More specifically, given the condition $\ell_{1}\ell_{2}<1$, it follows from the proof of Proposition~\ref{FiniteGainProperty} that $\bm{\Sigma}_1 \parallel_{\text{f}} \bm{\Sigma}_2$ and $\bm{\Sigma}_1' \parallel_{\text{f}} \bm{\Sigma}_2'$ satisfy the small gain theorem \cite[Theorem 3.4.1]{green2012linear}, \textit{i.e.,} the feedback interconnections $\bm{\Sigma}_1 \parallel_{\text{f}} \bm{\Sigma}_2$ and $\bm{\Sigma}_1' \parallel_{\text{f}} \bm{\Sigma}_2'$ are internally stable. As a result, subject to $\mathcal{L}_2$ signals, the output solution uniquely exists and belongs to the $\mathcal{L}_2$ space. Consequently, the $0$-asymptotic stability of the feedback connection is not necessary for the bound on the output deviation to hold.  
\end{remark}

\section{Illustrative Example}\label{IllustrativeSection}
To illustrate the notion of (\textgamma,\textdelta)-similarity, we consider the problem of component replacement in the electrical network depicted in Figure~\ref{DetailedStructure}. 
The component $S$ supplies $M$ which, in turn, feeds $D$ with a specific voltage. While the architecture of $M$ is given, we have no information regarding the architecture of $S$ and $D$. 

\begin{figure}[h]
	\centering
	\includegraphics[width=0.35\textwidth]{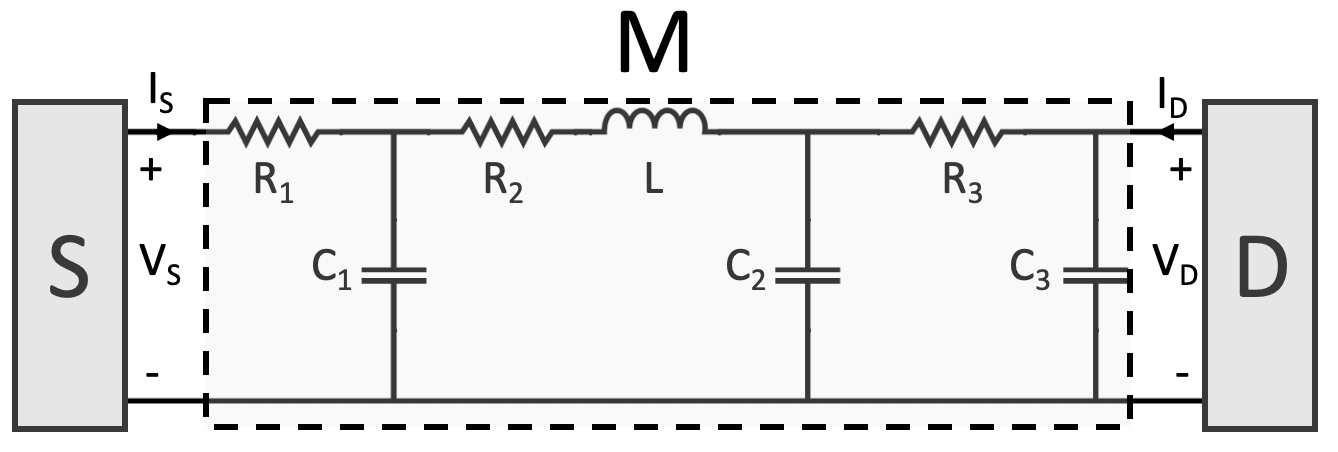}
	\caption{Supplied with the voltage $V_S$, $M$ feeds $D$ which, in turn, may inject the unwanted current $I_D$ into $M$. Note however that $V_S$ and $I_D$ are not known.}
	\label{DetailedStructure}
\end{figure} 
The detailed architecture of $M$ is depicted in Figure~\ref{DetailedStructure} with the parameter values given in Table~\ref{PrameterValues}. The component $S$ supplies $M$ with the voltage $V_S$. On the other hand, while being fed by $M$, the component $D$ may inject the unwanted current $I_D$, \textit{e.g.,} an inrush current, into $M$. It is worth mentioning that $V_S$ and $I_D$ are unknown. 

Suppose $M$ suffers a failure as its constituting elements reach the end of their life-time. We would like to replace $M$ with the component $M'$, depicted in Figure~\ref{ReplacingComponent} with parameter values given in Table~\ref{PrameterValues}. Such a component can be considerably cheaper as it comprises fewer elements. 
\begin{figure}[h]
	\centering
	\includegraphics[width=0.35\textwidth]{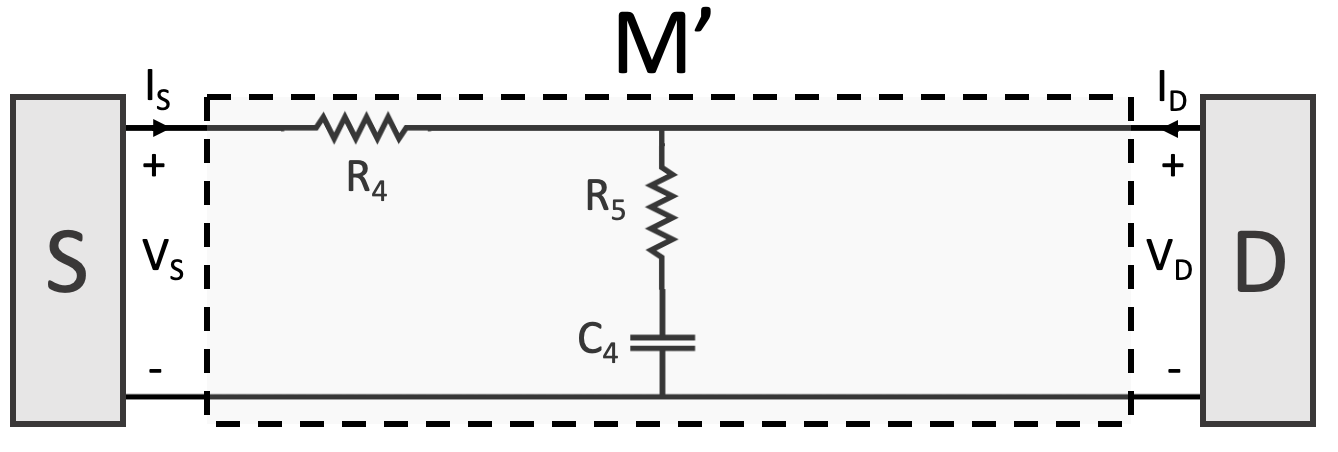}
	\caption{A candidate for replacing $M$, the component $M'$ comprises less elements.}
	\label{ReplacingComponent}
\end{figure} 
As a result of environmental noise, the voltages with which $M$ and $M'$ are supplied are similar, but not identical. On the other hand, the current that $D$ injects into $M$ is not necessarily the same as the one it injects into $M'$. Despite this, in order for the network to operate properly, $M'$ has to feed $D$ with almost the same voltage $M$ feeds $D$ with. This would be indeed the case if the input-output behavior of $M'$ is (approximately) contained in that of $M$. We will therefore use (\textgamma,\textdelta)-similarity to measure to what extent the input-output behavior of $M'$ is contained in that of $M$, see Remark~\ref{Application}. In particular, we desire $M'\preccurlyeq_{\gamma,\delta}M$ for some (hopefully small) $\gamma$ and $\delta$. 
	\begin{table}[ht]%
		\centering
		\caption{\footnotesize Parameter Values.}\label{PrameterValues}
		\begin{tabular*}{150pt}{@{\extracolsep\fill}ccc@{\extracolsep\fill}}	
			\toprule
			\textbf{Parameter} & \textbf{Value}&\textbf{Unit}\\
			\midrule
			$R_1$&10&$\Omega$ \\
			$R_2$&100&$\Omega$\\
			$R_3$&10&$\Omega$\\
			$R_4$&20&$\Omega$\\
			$R_5$&10&$\Omega$\\
			$C_1$&0.1&$F$\\
			$C_2$&0.3&$F$\\
			$C_3$&0.1&$F$\\
			$C_4$&0.5&$F$\\
			$L$&0.1&$H$\\
			\bottomrule
		\end{tabular*}
	\end{table}
Let $i_L$ denote the current passing through the inductor $L$ and $v_{C_i}$ denote the voltage across the capacitor $C_i$, for $i = 1,2,3,4$. Taking $x_1:= \operatorname{col}(v_{C_1},v_{C_2},v_{C_3},i_L)$, it follows from circuit theory that $M$ is governed by the dynamics
\begin{equation}\label{M_1}
\begin{aligned}
M: \begin{cases}
\dot{x}_1 = A_1x_1 + B_1u_1 + E_1d_1,\\
y_1 = C_1x_1,
\end{cases}
\end{aligned}
\end{equation}
where
\begin{equation*}
\begin{aligned}
A_1 &= \begin{bmatrix}
-\frac{1}{R_1C_1}&0&0&-\frac{1}{C_1}\\
0&-\frac{1}{R_3C_2}&\frac{1}{R_3C_2}&\frac{1}{C_2}\\
0&\frac{1}{R_3C_3}&-\frac{1}{R_3C_3}&0\\
\frac{1}{L}&-\frac{1}{L}&0&-\frac{R_2}{L}
\end{bmatrix}, \\
B_1 &= \begin{bmatrix}
\frac{1}{R_1C_1}\\0\\0\\0
\end{bmatrix},\;\;
C_1 = \begin{bmatrix}
0&1&0&0
\end{bmatrix}, \;\; E_1 = \begin{bmatrix}
0\\0\\\frac{1}{C_3}\\0
\end{bmatrix},
\end{aligned}
\end{equation*}
Here, $u_1$ and $d_1$ denote $V_S$ and $I_D$, respectively. In the same vein, taking $x_2 := v_{C_4}$, one obtains the dynamics of $M'$ as
\begin{equation}\label{M_2}
M': \begin{cases}
\dot{x}_{2} = A_{2}x_{2} + B_{2}u_{2} + E_{2}d_{2},\\
y_{2} = C_{2}x_{2},
\end{cases}
\end{equation}
with 
\begin{equation*}
\begin{aligned}
A_{2} &= -\frac{1}{C_4(R_4+R_5)}, \;\; B_{2} = \frac{1}{C_4(R_4+R_5)},\\
 E_{2} &= \frac{R_4}{C_4(R_4+R_5)},\quad\; C_{2} = 1. 
\end{aligned}
\end{equation*}
Once again, $u_2$ and $d_2$ represent $V_S$ and $I_D$, respectively. It is straightforward to see that the matrices $A_1$ and $A_2$ are Hurwitz. Moreover, the signals $u_1,u_2$ and $d_1,d_2$ are taken to be $\mathcal{L}_2$ signals.

We can now cast the problem into the framework of (\textgamma,\textdelta)-similarity according to Definition~\ref{Definition}. Because the voltages with which $M$ and $M'$ are supplied are subject to environmental noise, whose amplitude is almost negligible, the inputs $u_1$ and $u_{2}$ are almost the same. Due to the relatively small input dissimilarity, the parameter $\gamma$ does not have to be chosen very small. By contrast, $\delta$ has to be chosen very small, as $u_1$ and $u_2$ may have large amplitudes. For these reasons, we choose $\gamma=0.01$ and $\delta = 0.001$. Utilizing Theorem~\ref{GammaConformanceTheorem}, it turns out that $M' \preccurlyeq_{\gamma,\delta}M$ with the chosen values of $\gamma$ and $\delta$, which indicates that the input-output behavior of $M'$ is contained in that of $M$ to a large extent. This, as a result, implies that we may replace $M$ with $M'$ without significantly affecting the performance of the network, see Remark~\ref{Application}. Furthermore, in order to obtain \eqref{GammaConformanceInequality} with the smallest $\mu$, we minimized $\tilde{\mu}$ with respect to the constraint \eqref{FinalLMI}. The optimization problem has been solved in MATLAB R2015a using YALMIP \cite{Lofberg2004} and SDPT3 \cite{tutuncu2003solving}.

We compare the behavior of $M$ and $M'$ through numerical simulations, where  we take $u_1$, $u_{2}$, and $d_2$ as
\begin{equation*}
\begin{aligned}
u_1 &= \begin{cases}
20\sin(2\pi t), \;\;\; 0\leq t \leq 20,\\
0,
\end{cases}\\
u_{2} &= \begin{cases}
20.1\sin(2\pi t), \;\;\; 0\leq t \leq 20,\\
0,
\end{cases}\\
d_2 &= \frac{30}{0.2\sqrt{2\pi}}\exp\left(-\frac{(t-10)^2}{2(0.2)^2} \right).
\end{aligned}
\end{equation*} 
Clearly, while the components are fed over $[0,20]$, the component $M'$ experiences a surge in the current at $t = 10$. We run the simulation over $[0,50]$ which is large enough for the dynamics to reach their steady states. The output solutions as well as the output deviation are illustrated in Figure~\ref{OutputError}.
\begin{figure}[h]
	\centering
	\includegraphics[width=0.45\textwidth]{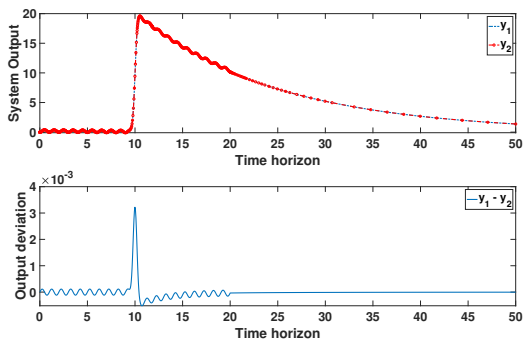}
	\caption{The output solutions of $M$ and $M'$ as well as the output deviation. Subject to a $d_2$ obtained from \eqref{FinalLMI}, see Figure~\ref{d2}, the system $M'$ reveals an output solution very similar to that of $M$. This indicates that for any solution trajectory of $M'$, there exists a trajectory of $M$ closely approximating it.}
	\label{OutputError}
\end{figure}
It is clear from Figure~\ref{OutputError} that when $d_1$ is chosen according to \eqref{FinalLMI}, which is depicted in Figure~\ref{d2}, the output solution of $M'$ is very similar to that of $M$. This, however, indicates the capability of $M$ in closely approximating the input-output behavior of $M'$. More precisely, for any solution trajectory of $M'$, there exists a trajectory of $M$ closely approximating it. This, as a result, indicates that one may replace $M$ with $M'$ without jeopardizing the performance of the network.
\begin{figure}[h]
	\centering
	\includegraphics[width=0.35\textwidth]{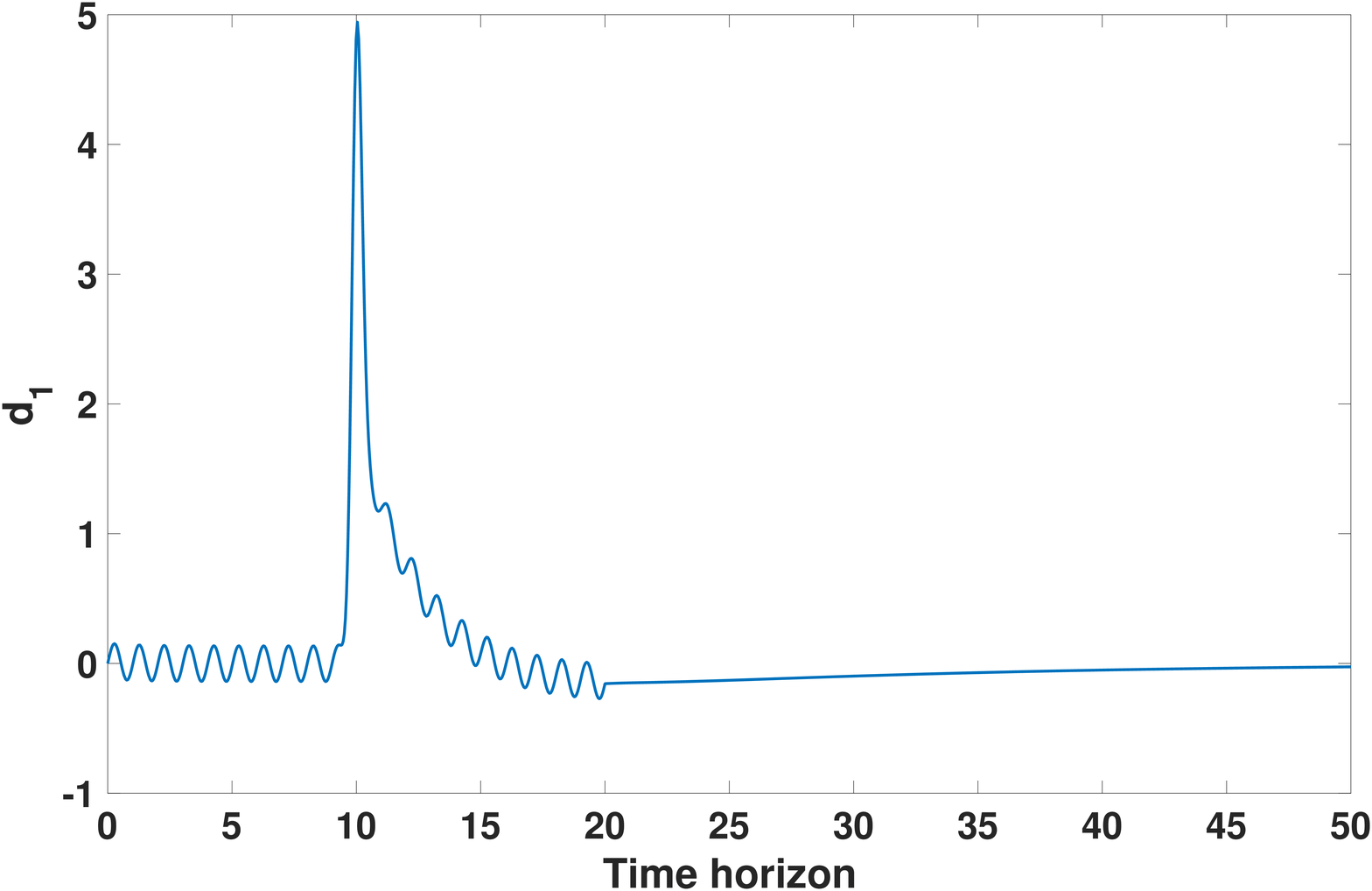}
	\caption{The disturbance $d_1$ subject to which the system $M$ reveals a solution trajectory that closely approximates the solution trajectory of $M'$.}
	\label{d2}
\end{figure}
\section{Conclusion}\label{ConclusionSection} 
This paper introduced (\textgamma,\textdelta)-similarity as a notion of system comparison, measuring to what extent the input-output behavior of two systems are similar. We formulated the behavioral similarity of two non-deterministic dynamical systems as the sensitivity of the difference between the output trajectories to the external inputs to the two systems, in the sense of $\mathcal{L}_2$ signal norms. We further characterized the notion of (\textgamma,\textdelta)-similarity as an LMI feasibility problem. We then studied (\textgamma,\textdelta)-similarity for deterministic systems and gave an interpretation of the notion in terms of transfer matrices. Subsequently, we  showed that the notion of (\textgamma,\textdelta)-similarity is preserved through series and feedback interconnections, which paves the way towards decomposing large-scale interconnected systems into interconnections of abstractions. 

For the future, our focus will be on utilizing this framework for control synthesis and compositional purposes such as abstraction and modular synthesis of interconnected systems. On the other hand, we aim to further extend this notion to nonlinear systems as well as systems whose inputs/disturbances are not necessarily $\mathcal{L}_2$ signals. 
\appendix
\subsection{Proof of Lemma~\ref{GammaConformanceFiniteTimeGammaConformance}}
Suppose \eqref{DissipativityRepresentation} holds for some $\varepsilon>0$. 
We will show that given any $T>0$, the DRE \eqref{DRE} admits a positive semi-definite solution over $[0,T]$, which by Lemma~\ref{GammaMuConformanceLemma}, implies that \eqref{GammaMuConformanceInequality} holds. To do so, we adopt the technique taken in the proof of Lemma~\ref{GammaMuConformanceLemma}. In particular, we show that given any $T>0$, \eqref{DRE} admits a solution over some interval $[T_1,T]$. Finally, we exploit the relation this solution has with $J_T$ to conclude the existence of the solution over $[0,T]$. 

Take any terminal time $T>0$; with the same reasoning as that of Lemma~\ref{GammaMuConformanceLemma}, there exists a $0\leq T_1\leq T$ such that \eqref{DRE} admits a solution $P$ over $[T_1,T]$.  
Taking an approach similar to that of \cite[Theorem 13.3]{trentelman2012control}, we observe 
\begin{equation*}
	\begin{aligned}
		&\mathscr{J}_T(0,x_s)\\
		&= \newsup_{w\in\mathcal{L}_2[0,T]}\inf_{d_2 \in \mathcal{L}_2[0,T]}  \left\{\norm{z}_{[0,T],R}^2 - \norm{w}_{[0,T],Q}^2 \right\}\\
		&\leq \newsup_{w\in\mathcal{L}_2[0,T]}\inf_{d_2 \in \mathcal{L}_2[0,T]} \Biggl( \norm{z}_{[0,T],R}^2 - \norm{w}_{[0,T],Q}^2\\
		&+\inf_{d_2 \in \mathcal{L}_2[T,\infty)}\biggl\{\int_{T}^{\infty}|z(t)|_{R}^2 \ \dt \left\vert \ \forall t> T: w(t)=0,\right.\\
		&\qquad\qquad\qquad \left.\lim\limits_{t\rightarrow\infty}x(t)=0 \biggr\}\right.\Biggr).
	\end{aligned}
\end{equation*}
As for the last inequality, it is noted that the term on the second line is bounded as a result of linear quadratic optimal control, \textit{e.g.,} \cite[Theorem 10.19]{trentelman2012control}. As a result, Choosing $d_2$ according to \eqref{DissipativityRepresentation}, one observes
	\begin{align}
		&\mathscr{J}_T(0,x_s)\nonumber\\
		&\leq \newsup_{w \in \mathcal{L}_2} \inf_{d_2 \in \mathcal{L}_2}\Bigl\{\norm{z}_{R}^2 - \norm{w}_Q^2 \left\vert \ \forall t > T:  w(t) = 0\Bigr\}\right.\nonumber\\
		&\leq\newsup_{w \in \mathcal{L}_2}\Bigl\{\norm{z(\,\cdot\,;0,x_s,d_2,w)}_{R}^2 - \norm{w}_Q^2 \left\vert  \forall t > T:   w(t) = 0\Bigr\}\right.\nonumber\\
		&\leq\newsup_{w \in \mathcal{L}_2}\Bigl\{\norm{z(\,\cdot\,;0,x_s,d_2,w)}_{R}^2 - \norm{w}_Q^2\Bigr\}.\label{BoundOnZeroInitialTime}
	\end{align}
On the other hand, the linearity of \eqref{CompositeSystem} implies that $z(t) = z(t;0,0,d_2,w) + z(t;0,x_s,0,0)$, by which \eqref{BoundOnZeroInitialTime} is written as
\begin{equation*}\label{Lemma1Aux1}
	\begin{aligned}
		&\mathscr{J}_T(0,x_s)\\
		&\leq\newsup_{w \in \mathcal{L}_2}\Bigl\{\norm{z(\,\cdot\,;0,0,d_2,w) + z(\,\cdot\,;0,x_s,0,0)}_{R}^2- \norm{w}_Q^2\}\\
		&\leq \newsup_{w \in \mathcal{L}_2}\Bigl\{\norm{z(\,\cdot\,;0,0,d_2,w)}_{R}^2 + \norm{z(\,\cdot\,;0,x_s,0,0)}_{R}^2- \norm{w}_Q^2\\
		&\qquad\qquad + 2\norm{z(\,\cdot\,;0,0,d_2,w)}_{R}\norm{ z(\,\cdot\,;0,x_s,0,0)}_{R}\Bigr\}.
	\end{aligned}
\end{equation*}
We now exploit \eqref{DissipativityRepresentation} to achieve
\begin{equation}\label{Lemma1Aux2}
	\begin{aligned}
		&\mathscr{J}_T(0,x_s)\\
		&\leq \newsup_{w \in \mathcal{L}_2}\Bigl\{-\varepsilon \norm{w}^2 + \norm{z(\,\cdot\,;0,x_s,0,0)}_{R}^2\\
		&\qquad + 2\norm{z(\,\cdot\,;0,0,d_2,w)}_{R}\norm{ z(\,\cdot\,;0,x_s,0,0)}_{R} \Bigr\}.
	\end{aligned}
\end{equation}
Once again, it follows from \eqref{DissipativityRepresentation} that
\begin{equation*}
	\norm{z(\,\cdot\,;0,0,d_2,w)}_{R} \leq (\lambda_{\text{max}}(Q)-\varepsilon)^{\frac{1}{2}}\norm{w},
\end{equation*}
which simplifies \eqref{Lemma1Aux2} as
\begin{equation}\label{Lemma1Aux3}
	\begin{aligned}
		&\mathscr{J}_T(0,x_s)\\
		&\leq\newsup_{w \in \mathcal{L}_2}\left\{ -\varepsilon\norm{w}^2 + \norm{z(\,\cdot\,;0,x_s,0,0)}_{R}^2\right.\\
		&\qquad\quad\left.+ 2(\lambda_{\text{max}}(Q)-\varepsilon)^{\frac{1}{2}}\norm{w}\norm{z(\,\cdot\,;0,x_s,0,0)}_{R}\right\}.
	\end{aligned}
\end{equation}
Since $A$ is Hurwitz, $z(\,\cdot\,;0,x_s,0,0) \in \mathcal{L}_2$ for all $x_s$ \cite[Theorem 3.1.1]{green2012linear}, implying that the last term, or equivalently, $\mathscr{J}_T(0,x_s)$ is bounded. As a result of \eqref{DRECostFunctionRelation}, this implies that the solution $P(\tau)$ is uniformly upper bounded over $[T_1,T]$. Once again, \eqref{DRECostFunctionRelation} and \eqref{CostFunctionMonotonicity} imply that for any $\tau \in [T_1,T]$, $P(\tau) \succcurlyeq 0$. Consequently, $P(\tau)$ is uniformly bounded over the interval $[T_1,T]$, which, by Theorem~10.7 of \cite{trentelman2012control}, implies that the solution $P$ exists over the whole interval $[0,T]$. 
\subsection{Proof of Theorem~\ref{GammaConformanceLemma}}
\textit{If part}: Suppose there exist constants $\eta,\mu>0$ such that, for matrices $Q$ and $R$ as in \eqref{QR}, such that \eqref{ARE} admits a positive semi-definite solution satisfying \eqref{Stabilizing}. As we would like to show that \eqref{DissipativityRepresentation} is satisfied, let $w\in\mathcal{L}_2$ be arbitrary. As a first step, we claim that there exists $d_2^* \in \mathcal{L}_2$ such that $x(\cdot) = x(\,\cdot\,;0,0,d_2^*,w)$ satisfies 
\begin{equation}\label{TheoremAux1}
\lim\limits_{t\rightarrow\infty}x(t) = 0.
\end{equation}
This, in turn, implies that 
\begin{equation}\label{TheoremAux2}
\int_{0}^{\infty}\frac{\dif x^T(t)Px(t)}{\dt} \ \dt=\lim\limits_{t\rightarrow\infty}x^T(t)Px(t) = 0,
\end{equation}
which will turn out to be useful later. Choose 
\begin{equation}\label{TheoremAux3}
d_2^* = Fx, \qquad F = -(D^TRD)^{-1}B^TP,
\end{equation}
and consider \eqref{CompositeSystem} subject to $d_2 = d_2^*$. If $A+BF$ is Hurwitz, it follows from \cite[Lemma 4.8]{brogliato2007dissipative} that \eqref{TheoremAux1} holds \emph{for any} $w \in \mathcal{L}_2$. We therefore proceed to show that $A+BF$ is Hurwitz. Given the fact that $D^TRC = 0$, \eqref{ARE} can be written as
\begin{equation}\label{AltARE}
\begin{aligned}
0=&\;(A+BF)^TP+P(A+BF) \\ 
&+(C+DF)^TR(C+DF)+ PEQ^{-1}E^TP.
\end{aligned}
\end{equation}
Let $v$ be an eigenvector of $A+BF$ with the eigenvalue $\lambda$. Multiplying \eqref{AltARE} by $v^*$ and $v$ from the left and right, respectively, we achieve
\begin{equation}\label{GammaConformanceAux5}
2\mathrm{Re}\{\lambda \}v^*Pv = -\left\vert E^TPv\right\vert_{Q^{-1}}^2 - \left\vert\left(C+DF\right)v\right\vert_R^2 \leq 0,
\end{equation} 
which suggests that either $Pv = 0$ or $\mathrm{Re}\{\lambda \}\leq0$. If the former holds, then
$$(A+BF)v = Av = \lambda v,$$
indicating $\mathrm{Re}\{\lambda \}<0$. Now, consider the case $\mathrm{Re}\{\lambda \}\leq 0$. Suppose $\mathrm{Re}\{\lambda \}=0$, which implies $E^TPv = 0$, and therefore,
\begin{equation*}
(A+BF)v = (A - B(D^TRD)^{-1}B^TP + EQ^{-1}E^TP)v = \lambda v,
\end{equation*}
which, by \eqref{Stabilizing}, implies $\mathrm{Re}\{\lambda \} < 0$, contradicting the initial assumption. Consequently, the matrix $A + BF$ is Hurwitz.

Now that we have shown that $A+BF$ is Hurwitz, we conclude that subject to $d_2^*$, for any $w\in\mathcal{L}_2$, \eqref{TheoremAux2} holds, which also implies that $d_2^*\in\mathcal{L}_2$. As a result, similar to the proof of Lemma \ref{GammaMuConformanceLemma}, we can write  
\begin{equation}\label{GammaConformanceAux3}
\begin{aligned}
&\lim\limits_{T\rightarrow\infty}J_T(0,0,d_2^*,w)\\
&= -\int_{0}^{\infty}\left\vert w(t)-Q^{-1}E^TP(t)x(t)\right\vert_Q^2\dt \leq 0.
\end{aligned}
\end{equation} 
As a result, it remains to show that there exists an $\varepsilon>0$ such that \eqref{DissipativityRepresentation} holds. For this purpose, define
\begin{equation}\label{VirtualInput2}
v := Q^{\frac{1}{2}}w - Q^{-\frac{1}{2}}E^TPx,
\end{equation}
and consider the dynamics of the composite system subject to $d_2 = d_2^*$. Similar to Lemma~\ref{GammaMuConformanceLemma}, after expressing $w$ in terms of $v$ using \eqref{VirtualInput2}, we obtain the system
\begin{align}
\dot{x} &= (A -B(D^TRD)^{-1}B^TP + EQ^{-1}E^TP)x + EQ^{-\frac{1}{2}}v,\nonumber\\
w&=Q^{-\frac{1}{2}}v + Q^{-1}E^TPx.\label{GammaConformanceAux4}
\end{align}
As a result of \eqref{Stabilizing}, the system \eqref{GammaConformanceAux4} is stable, implying that there exists an $\varepsilon>0$ such that $\norm{w} \leq \varepsilon^{-\frac{1}{2}}\norm{v}$, for trajectories corresponding to zero initial conditions.  This, however, leads to
\begin{equation}\label{TheoremAux5}
\lim\limits_{T\rightarrow\infty}J_T(0,0,d_2^*,w) = -\norm{v}^2 \leq -\varepsilon\norm{w}^2,
\end{equation}
for any $w \in \mathcal{L}_2$, which indicates that \eqref{DissipativityRepresentation} holds.
\\
\\
\textit{Only if} part: Suppose there exist a constant $\varepsilon$ and positive definite matrices $Q$ and $R$, as in \eqref{QR}, such that \eqref{DissipativityRepresentation} holds. 

Since \eqref{DissipativityRepresentation} holds, Lemma~\ref{GammaConformanceFiniteTimeGammaConformance} implies that \eqref{GammaMuConformanceInequality} holds for any $T>0$. Subsequently, Lemma~\ref{GammaMuConformanceLemma} guarantees that the DRE \eqref{DRE} admits a positive semi-definite solution over $[0,T]$. We denote this solution (at time $t\in[0,T]$) with $\bar{P}(t,T,0)$, emphasizing the fact that the solution exists over $[0,T]$ with a zero terminal condition. As $\bar{P}(t,T,0)$ exists for any $T>0$, the DRE \eqref{DRE} admits a positive semi-definite solution over $[0,\infty)$. 

We show that $\bar{P}(t,T,0)$ converges to a constant limit satisfying \eqref{ARE}, as $T$ grows to infinity. For this purpose, we initially recall two main features of $\bar{P}(t,T,0)$, namely uniform boundedness and monotonicity. To show the former, we recall that given any $T > 0$, \eqref{DRECostFunctionRelation} holds
for any $\tau \in [0,T]$. Moreover, as shown in Lemma~\ref{GammaConformanceFiniteTimeGammaConformance}, \eqref{DissipativityRepresentation} implies that $\mathscr{J}_T(0,x_s)$ is bounded, which, along with \eqref{CostFunctionMonotonicity}, implies that for any given $T > 0$, the solution $\bar{P}(t,T,0)$ is uniformly bounded over $[0,T]$. More importantly, the lower and upper bounds of $\bar{P}(t,T,0)$ are \emph{independent} from $T$. It directly follows from \eqref{CostFunctionMonotonicity} that $\bar{P}(t,T,0)$ is monotonically non-increasing as a function of $t$. However, since the system is time-invariant, $\bar{P}(t,T,0) = \bar{P}(\tau,T-t+\tau,0)$, implying that $\bar{P}(t,T,0)$ is a monotonically non-decreasing function of $T$ (see also \cite[Lemma 6.3.2]{green2012linear}). Now, because $\bar{P}(t,T,0)$ is uniformly bounded and monotonically \emph{non-decreasing} with respect to $T$, it has a limit $P_t$ as $T$ grows to infinity. However, as a result of time-invariance, one concludes that for all $t$, 
\begin{equation*}
P_t = \lim\limits_{T\rightarrow \infty}\bar{P}(t,T,0) = \lim\limits_{T\rightarrow\infty}\bar{P}(0,T-t,0) = \lim\limits_{T\rightarrow\infty}\bar{P}(0,T,0),
\end{equation*}
which acknowledges that $P_t$ is constant \cite[Lemma 6.3.3]{green2012linear}. For this reason, we drop the subscript from now on.  Obviously, $P$ is positive semi-definite. To show that $P$ satisfies \eqref{ARE}, we take the approach of \cite[Lemma 6.3.3]{green2012linear}. As the system is time invariant, $\bar{P}(t,T,0)$ only depends on $T-t$. Therefore, for any $T^* \in [0,\infty)$, 
\begin{equation*}
\begin{aligned}
\forall t \leq T^*:P &= \lim\limits_{T\rightarrow \infty}\bar{P}(t,T,0)\\ &= \lim\limits_{T\rightarrow\infty}\bar{P}\Bigl(t,T^*,\bar{P}(T^*,T,0)\Bigr)\\ &= \bar{P}\Bigl(t,T^*,\lim\limits_{T\rightarrow\infty}\bar{P}(T^*,T,0)\Bigr)\\ &= \bar{P}(t,T^*,P),
\end{aligned}
\end{equation*}
by which one concludes that for any $T^*\in[0,\infty)$, $P$ satisfies the DRE \eqref{DRE} over $[0,T^*]$ with a terminal condition equal to $P$. As $P$ is constant, this implies that $P$ satisfies the ARE \eqref{ARE}.

Now, we show that $P$ satisfies \eqref{Stabilizing}. As an intermediate step, we show that $A - B(D^TRD)^{-1}B^TP$ is Hurwitz. Consider the matrix $F$ defined as in \eqref{TheoremAux3} and rewrite the ARE \eqref{ARE} as \eqref{AltARE}. Let $v$ be an eigenvector of $A+BF$ with eigenvalue $\lambda$. We respectively multiply \eqref{AltARE} with $v^*$ and $v$ from the left and right to achieve \eqref{GammaConformanceAux5}.
Clearly, either $Pv = 0$ or $\mathrm{Re}\{\lambda \}\leq0$. For the former case, we conclude that
$\mathrm{Re}\{\lambda\}<0$ as before. Hence, consider the case $\mathrm{Re}\{\lambda \}\leq0$. Clearly, if $\mathrm{Re}\{\lambda \}=0$, then
\begin{equation}\label{GammaConformanceAux2}
(A+BF)v = \lambda v, \;\; (C+DF)v=0.
\end{equation}
Nonetheless, as $D$ is injective, \eqref{GammaConformanceAux2} implies that $\lambda$ is a zero of the system $(A,B,C,D)$ \cite[Exercise 13.3]{trentelman2012control}, or equivalently, $\lambda$ is an unobservable eigenvalue of the pair \cite[Exercise 10.1]{trentelman2012control}
\begin{equation*}
\left(A -B\left(DD^T\right)^{-1}D^TC, C - D\left(DD^T\right)^{-1}D^TC\right),
\end{equation*}
which, because $D^TC = 0$, means that $\lambda$ is an unobservable eigenvalue of $(A,C)$. However, as $A$ is Hurwitz, the pair $(A,C)$ is detectable, and therefore, $\mathrm{Re}\{\lambda \}< 0$ contradicting the assumption that $\lambda$ lies on the imaginary axis. Consequently, the matrix $A + BF$ is Hurwitz. 

We now proceed to the next step, \textit{i.e.,} we show that \eqref{Stabilizing} holds. For this purpose, take the disturbance $w$ generated by the autonomous system 
\begin{equation}\label{Aux2}
\begin{aligned}
\dot{\xi}(t) &= \Bigl(A - B(D^TRD)^{-1}B^TP + EQ^{-1}E^TP\Bigr)\xi(t),\\
w(t) &= Q^{-1}E^TP\xi(t),
\end{aligned}
\end{equation}
with $\xi(0) = x_0$. It is not difficult to notice that the dynamics of \eqref{Aux2} can also be written as
\begin{equation}\label{TheoremAux4}
\dot{\xi}(t) = (A+BF)\xi + Ew.
\end{equation}
We have shown earlier that $A+BF$ is Hurwitz. If we show that $w \in \mathcal{L}_2$, then we may conclude that $\lim\limits_{t\rightarrow\infty}\xi(t) =0$, for any $\xi(0)$. This, however, would imply the asymptotic stability of \eqref{Aux2}, which, in turn, implies \eqref{Stabilizing}. We therefore proceed to show that $w \in \mathcal{L}_2$. With the aim of doing so, consider \eqref{CompositeSystem} subject to the input $d_2 = Fx$ and the disturbance $w = w_T(t) := Q^{-1}E^T\bar{P}(t,T,0)x(t)$. Corresponding to $w_T$, we also define the signal
\begin{equation*}
\bar{w}_{T}(t) := \begin{cases}
w_T(t), \qquad t\leq T,\\
0, \qquad\qquad t>T.
\end{cases}
\end{equation*}
Accordingly, one finds
\begin{equation*}
\begin{aligned}
&J_T(0,x_0,Fx, w_{T})\\
&= \norm{z(\,\cdot\,;0,x_0,Fx,w_T)}_{[0,T],R}^2 - \norm{w_T}_{[0,T],Q}^2\\
&= \norm{z(\,\cdot\,;0,x_0,Fx,\bar{w}_T)}_{[0,T],R}^2 - \norm{\bar{w}_T}_{[0,T],Q}^2\\
&\leq \norm{z(\,\cdot\,;0,x_0,Fx,\bar{w}_T)}_{[0,T],R}^2 - \norm{\bar{w}_T}_{[0,T],Q}^2\\
&\quad+\int_{T}^{\infty}|z(t;0,x_0,Fx,\bar{w}_T)|_R^2 \ \dt.
\end{aligned}
\end{equation*}
Note that because $A+BF$ is Hurwitz, the indefinite integral is well-defined. It now follows from $\norm{\bar{w}_T}_Q = \norm{\bar{w}_T}_{[0,T],Q}$ that
\begin{equation*}
\begin{aligned}
&J_T(0,x_0,Fx, w_{T})\\
& \leq \norm{z(\,\cdot\,;0,x_0,Fx,\bar{w}_T)}_{R}^2 - \norm{\bar{w}_T}_{Q}^2\\
&\leq \norm{z(\,\cdot\,;0,0,Fx,\bar{w}_T) + z(\,\cdot\,;0,x_0,0,0)}_{R}^2 - \norm{\bar{w}_T}_{Q}^2\\
&\leq \norm{z(\,\cdot\,;0,0,Fx,\bar{w}_T)}_{R}^2 + \norm{z(\,\cdot\,;0,x_0,0,0)}_{R}^2 \\&\quad+2\norm{z(\,\cdot\,;0,0,Fx,\bar{w}_T)}_{R}\norm{z(\,\cdot\,;0,x_0,0,0)}_{R}- \norm{\bar{w}_T}_{Q}^2.
\end{aligned}
\end{equation*}
Now, a similar reasoning as in obtaining \eqref{Lemma1Aux2} and \eqref{Lemma1Aux3} gives
\begin{equation*}
\begin{aligned}
&J_T(0,x_0,Fx,w_T)\\
&\leq -\varepsilon\norm{\bar{w}_T}^2 + \norm{z(\,\cdot\,;0,x_0,0,0)}_{R}^2 \\
&\;\;\;+2(\lambda_{\text{max}}(Q)-\varepsilon)^{\frac{1}{2}}\norm{\bar{w}_T}\norm{z(\,\cdot\,;0,x_0,0,0)}_R.
\end{aligned}
\end{equation*}
However, we recall that $\norm{\bar{w}_T} = \norm{\bar{w}_T}_{[0,T]} = \norm{w_T}_{[0,T]}$. Therefore, we conclude
\begin{equation*}
\begin{aligned}
J_T(0,x_0,Fx,w_T)\leq-\varepsilon(\norm{w_T}_{[0,T]}-c_1)^2 + c_2
\end{aligned}
\end{equation*}
where the constants
\begin{equation*}
\begin{aligned}
c_1 &= \frac{(\lambda_{\text{max}}(Q)-\varepsilon)^{1/2}}{\varepsilon}\norm{z(\,\cdot\,;0,x_0,0,0)}, \\
c_2 &= \frac{\lambda_{\text{max}}(Q)}{\varepsilon}\norm{z(\,\cdot\,;0,x_0,0,0)}^2,
\end{aligned}
\end{equation*}
are independent of $T$. On the other hand, it follows from \eqref{DRECostFunctionRelation} that
\begin{equation*}
J_T(0,x_0,Fx,w_{T}) \geq x_0^T\bar{P}(0,T,0)x_0,
\end{equation*}
which, along with the fact that $\bar{P}(0,T,0) \succcurlyeq 0$, implies that $\norm{w_{T}}_{[0,T]}$ is uniformly upper bounded with respect to $T$. Obviously, for all $t^* \leq T$, one finds $\norm{w_{T}}_{[0,t^*]}$ uniformly upper bounded (with respect to $t^*$ and $T$) as well. Since $$\lim\limits_{T\rightarrow\infty}\bar{P}(t,T,0) = P,$$ one concludes $\lim\limits_{T \rightarrow \infty}w_{T}(t) = w(t)$ for all $t$, implying that for all $t^*$,
$$\norm{w}_{[0,t^*]} = \lim\limits_{T \rightarrow \infty }\norm{w_{T}}_{[0,t^*]},$$
is uniformly upper bounded with respect to $t^*$ and eventually guaranteeing $w^* \in \mathcal{L}_2$. As pointed out earlier, this completes the proof. More specifically, because $A +BF$ is Hurwitz and $Ew^* \in \mathcal{L}_2$, one has  $\lim\limits_{t\rightarrow\infty}\xi(t) = 0,$ for all $\xi(0)$. This implies that \eqref{TheoremAux4} is $0$-asymptotically stable, or equivalently \eqref{Aux2}, is asymptotically stable and therefore, \eqref{Stabilizing} is satisfied. 
\subsection{Proof of Proposition~\ref{FiniteGainProperty}}
As a direct consequence of $0$-asymptotic stability, $\bm{\Sigma}_1$ has finite $\mathcal{L}_2$ gain, see, \textit{e.g.,} \cite[Section 4.1]{desoer2009feedback} and \cite[Lemma 4.8]{brogliato2007dissipative}. Hence, there exist $\ell_1,k_1>0$ such that for all $u_1,d_1\in\mathcal{L}_2$,
\begin{equation}\label{PropositionFiniteGainAux1}
\norm{y_1}^2 \leq \ell_1\norm{u_1}^2 + k_1\norm{d_1}^2 \leq \ell_1\norm{\begin{bmatrix}
	u_1\\u_2
	\end{bmatrix}}^2 + k_1\norm{d_1}^2.
\end{equation}
Needless to say, it follows from Proposition~\ref{IndividualDetereminism} that $\ell_1 < \gamma+\delta$. On the other hand, $\bm{\Sigma}_1 \preccurlyeq_{\gamma,\delta} \bm{\Sigma}_2$ implies the existence of $\varepsilon,\eta,\mu>0$ and $d_2\in\mathcal{L}_2$ such that \eqref{GammaConformanceInequality} holds. Choosing $d_2$ according to \eqref{GammaConformanceInequality}, we add \eqref{PropositionFiniteGainAux1} to \eqref{GammaConformanceInequality} and utilize the triangular and the Cauchy-Schwartz inequalities to obtain
\begin{equation*}
\begin{aligned}
\norm{y_2}^2 \leq&\; 2\gamma\norm{u_1-u_2}^2+ 2(\ell_1+\delta)\norm{\begin{bmatrix}
	u_1\\u_2
	\end{bmatrix}}^2\\
& + 2(\mu+k_1)\norm{d_1}^2\\
=& \norm{\begin{bmatrix}
	u_1\\u_2
	\end{bmatrix}}_{\hat{Q}}^2 + 2(\mu+k_1)\norm{d_1}^2,
\end{aligned}
\end{equation*}
where
\begin{equation*}
\hat{Q} = \begin{bmatrix}
2(\gamma+\delta+\ell_1)I&-2\gamma I\\
-2\gamma I&2(\gamma+\delta+\ell_1) I
\end{bmatrix}.
\end{equation*}
Defining $\ell_2 := \lambda_{\text{max}}(\hat{Q})$ and $k_2 := 2(\mu+k_1)$, this yields
\begin{equation}\label{PropositionFiniteGainAux2}
\norm{y_2}^2 \leq \ell_2\norm{\begin{bmatrix}
	u_1\\u_2
	\end{bmatrix}}^2 + k_2\norm{d_1}^2.
\end{equation}
It is now clear that for the choice of $\ell = \ell_1+\ell_2$ and $k=k_1+k_2$, \eqref{FiniteGain} follows from \eqref{PropositionFiniteGainAux1} and \eqref{PropositionFiniteGainAux2}.
\subsection{Proof of Proposition~\ref{SeriesGammaConformance}}
Let $u_1,u_1'\in\mathcal{L}_2$ and $d_1,d_2\in\mathcal{L}_2$ be given. It follows from  $\bm{\Sigma}_1 \preccurlyeq_{\gamma_{1},\delta_{1}} \bm{\Sigma}_{1}'$ that there exist $\varepsilon_{1},\eta_{1},\mu_{1}>0$ and $d_1'\in\mathcal{L}_2$ such that 
\begin{equation}\label{PropositionSeriesAux1}
\begin{aligned}
\norm{y_1-y_1'}^2 \leq&\; \gamma_{1}\norm{u_1-u_1'}^2 + (\delta_{1}-\varepsilon_{1})\norm{\begin{bmatrix}
	u_1\\u_1'
	\end{bmatrix}}^2\\
& + (\mu_{1}-\varepsilon_{1})\norm{d_1}^2 - \eta_{1}\norm{d_1'}^2.
\end{aligned}
\end{equation}
Choosing $u_2 = y_1$ and $u_2' = y_1'$, $\bm{\Sigma}_2 \preccurlyeq_{\gamma_{2},\delta_{2}} \bm{\Sigma}_{2}'$ implies the existence of 
$\varepsilon_{2},\eta_{2},\mu_{2}>0$ and $d_2'\in\mathcal{L}_2$ such that 
\begin{equation}\label{PropositionSeriesAux2}
\begin{aligned}
\norm{y_2-y_2'}^2 \leq&\; \gamma_{2}\norm{y_1-y_1'}^2 + (\delta_{2}-\varepsilon_{2})\norm{\begin{bmatrix}
	y_1\\y_1'
	\end{bmatrix}}^2\\
& + (\mu_{2}-\varepsilon_{2})\norm{d_2}^2 - \eta_{2}\norm{d_2'}^2.
\end{aligned}.
\end{equation}
Choosing $d_1'$ according to \eqref{PropositionSeriesAux1}, it follows from Proposition~\ref{FiniteGainProperty} that there exist $\ell_{1},k_{1}>0$ such that
\begin{equation*}
\norm{\begin{bmatrix}
	y_1\\y_1'
	\end{bmatrix}}^2\leq \ell_{1}\norm{\begin{bmatrix}
	u_1\\u_1'
	\end{bmatrix}}^2 + k_{1}\norm{d_1}^2.
\end{equation*}
It now follows from \eqref{PropositionSeriesAux1} and \eqref{PropositionSeriesAux2} that
\begin{equation*}
\begin{aligned}
\norm{y_2-y_2'}^2 \leq&\;\gamma_{1}\gamma_{2}\norm{u_1-u_1'}^2\\
&+\Bigl(\gamma_{2}(\delta_{1}-\varepsilon_{1}) + \ell_{1}(\delta_{2}-\varepsilon_{2}) \Bigr)\norm{\begin{bmatrix}
	u_1\\u_1'
	\end{bmatrix}}^2\\
&+ \Bigl(\gamma_{2}(\mu_{1}-\varepsilon_{1})+k_{1}(\delta_{2}-\varepsilon_{2}) \Bigr)\norm{d_1}^2\\
&+ (\mu_{2}-\varepsilon_{2})\norm{d_2}^2- \gamma_{2}\eta_{1}\norm{d_1'}^2 - \eta_{2}\norm{d_2'}^2,
\end{aligned}
\end{equation*}
which, by Definition~\ref{Definition}, guarantees that $\bm{\Sigma}_1 \parallel_{\text{s}} \bm{\Sigma}_2 \preccurlyeq_{\gamma_s,\delta_s} \bm{\Sigma}_{1}' \parallel_{\text{s}} \bm{\Sigma}_{2}'$ with
\begin{equation*}
\begin{aligned}
\gamma_s = \gamma_{1}\gamma_{2}, \;\; \delta_s = \gamma_{2}\delta_{1} + \delta_{2}\ell_{1},
\end{aligned}
\end{equation*}
which completes the proof.
\subsection{Proof of Proposition~\ref{FeedbackGammaConformance}}
Let $e_1,e_1',e_2,e_2'\in\mathcal{L}_2$ and $d_1,d_2\in\mathcal{L}_2$ be given. Choosing $u_1 = e_1 - y_2$ and $u_1' = e_1' - y_2'$, it follows from  $\bm{\Sigma}_1 \preccurlyeq_{\gamma_{1},\delta_{1}} \bm{\Sigma}_{1}'$ that there exist $\varepsilon_{1},\eta_{1},\mu_{1}>0$ and $d_1'\in\mathcal{L}_2$ such that 
\begin{equation*}
\begin{aligned}
\norm{y_1-y_1'}^2 \leq&\; \gamma_{1}\norm{e_1-y_2-e_1'+y_2'}^2\\
& + (\delta_{1}-\varepsilon_{1})\norm{\begin{bmatrix}
	e_1-y_2\\e_1'-y_2'
	\end{bmatrix}}^2\\
& + (\mu_{1}-\varepsilon_{1})\norm{d_1}^2 - \eta_{1}\norm{d_1'}^2.
\end{aligned}
\end{equation*}
Using the triangular and Cauchy-Schwartz inequalities, this leads to
\begin{equation}\label{PropositionFeedbackAux1}
\begin{aligned}
\norm{y_1-y_1'}^2 \leq&\; \frac{\tilde{\epsilon}}{\epsilon}\gamma_{1}\norm{e_1-e_1'}^2 + 2(\delta_{1}-\varepsilon_{1})\norm{\begin{bmatrix}
		e_1\\e_1'
\end{bmatrix}}^2 \\
&+ \tilde{\epsilon}\gamma_{1}\norm{y_2-y_2'}^2  + 2(\delta_{1}-\varepsilon_{1})\norm{\begin{bmatrix}
	y_2\\y_2'
	\end{bmatrix}}^2\\
& + (\mu_{1}-\varepsilon_{1})\norm{d_1}^2 - \eta_{1}\norm{d_1'}^2,
\end{aligned}
\end{equation}
where $\tilde{\epsilon} = 1+\epsilon$. Completely analogously, we can use $\bm{\Sigma}_2 \preccurlyeq_{\gamma_2,\delta_2} \bm{\Sigma}_2'$ to conclude that there exists $d_2'\in\mathcal{L}_2$ such that we can bound $\norm{y_2-y_2'}^2$ in a similar way as \eqref{PropositionFeedbackAux1}. Substitution of this bound in \eqref{PropositionFeedbackAux1} and rearranging terms (hereby using \eqref{SmallGain}) leads to
	\begingroup
	\allowdisplaybreaks
			\begin{align}
				&\norm{y_1-y_1'}^2 \nonumber\\
				&\leq\frac{\tilde{\epsilon}\gamma_{1}}{\epsilon	(1-\tilde{\epsilon}^2\gamma_{1}\gamma_{2})}\norm{e_1-e_1}^2+\frac{\tilde{\epsilon}^2\gamma_{1}\gamma_{2}}{\epsilon(1-\tilde{\epsilon}^2\gamma_{1}\gamma_{2})}\norm{e_2-e_2'}^2\nonumber\\
				&+\frac{2(\delta_{1}-\varepsilon_{1})}{1-\tilde{\epsilon}^2\gamma_{1}\gamma_{2}}\norm{\begin{bmatrix}
						e_1\\e_1'
				\end{bmatrix}}^2 + \frac{2\tilde{\epsilon}\gamma_{1}(\delta_{2}-\varepsilon_{2})}{1-\tilde{\epsilon}^2\gamma_{1}\gamma_{2}}\norm{\begin{bmatrix}
				e_2\\e_2'
			\end{bmatrix}}^2\nonumber\\&+\frac{2\tilde{\epsilon}\gamma_{1}(\delta_{2}-\varepsilon_{2})}{1-\tilde{\epsilon}^2\gamma_{1}\gamma_{2}}\norm{\begin{bmatrix}
						y_1\\y_1'
				\end{bmatrix}}^2++\frac{2(\delta_{1}-\varepsilon_{1})}{1-\tilde{\epsilon}^2\gamma_{1}\gamma_{2}}\norm{\begin{bmatrix}
				y_2\\y_2'
			\end{bmatrix}}^2\nonumber\\ & +\frac{(\mu_{1}-\varepsilon_{1})}{1-\tilde{\epsilon}^2\gamma_{1}\gamma_{2}}\norm{d_1}^2 + \frac{\tilde{\epsilon}\gamma_{1}(\mu_{2}-\varepsilon_{2})}{1-\tilde{\epsilon}^2\gamma_{1}\gamma_{2}}\norm{d_2}^2 \nonumber\\
				& -\frac{\eta_{1}}{1-\tilde{\epsilon}^2\gamma_{1}\gamma_{2}}\norm{d_1'}^2- \frac{\tilde{\epsilon}\gamma_{1}\eta_{2}}{1-\tilde{\epsilon}^2\gamma_{1}\gamma_{2}}\norm{d_2'}^2.\label{PropositionFeedbackAux3}
			\end{align}
		\endgroup
On the other hand, choosing $d_1'$ according to \eqref{PropositionFeedbackAux1}, one exploits Proposition~\ref{FiniteGainProperty} to obtain
\begingroup
\allowdisplaybreaks
\begin{align*}
\norm{\begin{bmatrix}
	y_1\\y_1'
	\end{bmatrix}}^2 \leq&\;\ell_{1}\norm{
	\begin{bmatrix}
	e_1-y_2\\e_1'-y_2'
	\end{bmatrix}}^2 + k_{1}\norm{d_1}^2\\
\leq&\; \frac{\tilde{\epsilon}}{\epsilon}\ell_{1}\norm{\begin{bmatrix}
	e_1\\e_1'
	\end{bmatrix}}^2 + \tilde{\epsilon}\ell_{1}\norm{\begin{bmatrix}
	y_2\\y_2'
	\end{bmatrix}}^2 + k_{1}\norm{d_1}^2\\
\leq&\; \frac{\tilde{\epsilon}}{\epsilon}\ell_{1}\norm{\begin{bmatrix}
	e_1\\e_1'
	\end{bmatrix}}^2 + \tilde{\epsilon}\ell_{1}\ell_{2}\norm{\begin{bmatrix}
	e_2+y_1\\e_2'+y_1'
	\end{bmatrix}}^2\\
&+ \tilde{\epsilon}\ell_{1}k_{2}\norm{d_2}^2 + k_{1}\norm{d_1}^2\\
\leq&\; \frac{\tilde{\epsilon}}{\epsilon}\ell_{1}\norm{\begin{bmatrix}
	e_1\\e_1'
	\end{bmatrix}}^2 + \frac{\tilde{\epsilon}^2}{\epsilon}\ell_{1}\ell_{2}\norm{\begin{bmatrix}
	e_2\\e_2'
	\end{bmatrix}}^2\\
&+ \tilde{\epsilon}^2\ell_{1}\ell_{2}\norm{\begin{bmatrix}
	y_1\\y_2
	\end{bmatrix}}^2+ \tilde{\epsilon}\ell_{1}k_{2}\norm{d_2}^2 + k_{1}\norm{d_1}^2.
\end{align*}
\endgroup
However, given \eqref{SmallGain}, it is easily seen that
\begingroup
\allowdisplaybreaks
\begin{align}
\norm{\begin{bmatrix}
	y_1\\y_1'
	\end{bmatrix}}^2 \leq&\; \frac{\frac{\tilde{\epsilon}}{\epsilon}\ell_{1}}{1-\tilde{\epsilon}^2\ell_{1}\ell_{2}}\norm{\begin{bmatrix}
	e_1\\e_1'
	\end{bmatrix}}^2 + \frac{\frac{\tilde{\epsilon}^2}{\epsilon}\ell_{1}\ell_{2}}{1-\tilde{\epsilon}^2\ell_{1}\ell_{2}}\norm{\begin{bmatrix}
	e_2\\e_2'
	\end{bmatrix}}^2\nonumber\\
& + \frac{k_{1}}{1-\tilde{\epsilon}^2\ell_{1}\ell_{2}}\norm{d_1}^2 + \frac{\tilde{\epsilon}\ell_{1}k_{2}}{1-\tilde{\epsilon}^2\ell_{1}\ell_{2}}\norm{d_2}^2, \label{PropositionFeedbackAux4}
\end{align}
\endgroup
whereas a similar bound can be found on $\operatorname{col}(y_2,y_2')$. Substitution of those bounds in \eqref{PropositionFeedbackAux3} leads to
\begingroup
\allowdisplaybreaks
\begin{align}
\norm{y_1-y_1'}^2 \leq&\; \frac{\frac{\tilde{\epsilon}}{\epsilon}\gamma_{1}}{1-\tilde{\epsilon}^2\gamma_{1}\gamma_{2}}\norm{e_1-e_1'}^2+\frac{\frac{\tilde{\epsilon}^2}{\epsilon}\gamma_{1}\gamma_{2}}{1-\tilde{\epsilon}^2\gamma_{1}\gamma_{2}}\norm{e_2-e_2'}^2\nonumber\\
&+\frac{2\frac{\tilde{\epsilon}^2}{\epsilon}\gamma_{1}\ell_{1}(\delta_{2}-\varepsilon_{2})}{\left(1-\tilde{\epsilon}^2\gamma_{1}\gamma_{2}\right)\left(1-\tilde{\epsilon}^2\ell_{1}\ell_{2}\right)}\norm{\begin{bmatrix}
	e_1\\e_1'
	\end{bmatrix}}^2\nonumber\\
&+\frac{2(\delta_{1}-\varepsilon_{1})\left(1 + \frac{\tilde{\epsilon}^2(1-\epsilon)}{\epsilon}\ell_1\ell_2 \right)}{\left(1-\tilde{\epsilon}^2\gamma_{1}\gamma_{2}\right)\left(1-\tilde{\epsilon}^2\ell_{1}\ell_{2}\right)}\norm{\begin{bmatrix}
	e_1\\e_1'
	\end{bmatrix}}^2\nonumber\\
&+\frac{2\frac{\tilde{\epsilon}}{\epsilon}\ell_{2}(\delta_{1}-\varepsilon_{1})  }{\left(1-\tilde{\epsilon}^2\gamma_{1}\gamma_{2}\right)\left(1-\tilde{\epsilon}^2\ell_{1}\ell_{2}\right)}\norm{\begin{bmatrix}
	e_2\\e_2'
	\end{bmatrix}}^2\nonumber\\
&+\frac{2\tilde{\epsilon}\gamma_{1}(\delta_{2}-\varepsilon_{2})\left(1 + \frac{\tilde{\epsilon}^2(1-\epsilon)}{\epsilon}\ell_1\ell_2\right)}{\left(1-\tilde{\epsilon}^2\gamma_{1}\gamma_{2}\right)\left(1-\tilde{\epsilon}^2\ell_{1}\ell_{2}\right)}\norm{\begin{bmatrix}
	e_2\\e_2'
	\end{bmatrix}}^2\nonumber\\
&+(\tilde{\mu}_{1}-\tilde{\varepsilon}_{1})\norm{\begin{bmatrix}
	d_1\\d_2
	\end{bmatrix}}^2-\tilde{\eta}_{1}\norm{\begin{bmatrix}
	d_1'\\d_2'
	\end{bmatrix}}^2, \label{PropositionFeedbackAux6}
\end{align}
\endgroup
for some constants $\tilde{\varepsilon}_{1},\tilde{\eta}_{1},\tilde{\mu}_{1}>0$.
A similar procedure can be applied for $\norm{y_2 - y_2'}^2$, such that it is clear from Definition~\ref{Definition} that $\bm{\Sigma}_1 \parallel_{\text{f}} \bm{\Sigma}_2 \preccurlyeq_{\gamma_f,\delta_f} \bm{\Sigma}_{1}' \parallel_{\text{f}} \bm{\Sigma}_{2}'$ with
\begingroup
\allowdisplaybreaks
\begin{align*}
\gamma_f &= \frac{\frac{\tilde{\epsilon}}{\epsilon}\max\{\gamma_{1},\gamma_{2} \} + \frac{\tilde{\epsilon}^2}{\epsilon}\gamma_{1}\gamma_{2}}{1-\gamma_{1}\gamma_{2}\tilde{\epsilon}^2},\\
\delta_f &= \frac{2\frac{\tilde{\epsilon}}{\epsilon}\max\{\delta_1, \delta_2 \}\max\{\ell_1,\ell_2\}\left(1+\tilde{\epsilon}\max\{\gamma_1,\gamma_2\} \right) }{\left(1-\tilde{\epsilon}^2\gamma_{1}\gamma_{2})(1-\tilde{\epsilon}^2\ell_{1}\ell_{2}\right)}\\
&\;\;\;+\frac{2\max\{\delta_1, \delta_2 \}\left(1+\tilde{\epsilon}\max\{\gamma_1,\gamma_2\} \right)\left(1+\frac{\tilde{\epsilon}^2(1-\epsilon)}{\epsilon}\ell_1\ell_2 \right)}{\left(1-\tilde{\epsilon}^2\gamma_{1}\gamma_{2})(1-\tilde{\epsilon}^2\ell_{1}\ell_{2}\right)},
\end{align*}
\endgroup
which proves the desired result. 
\bibliographystyle{ieeetr}
\bibliography{Reference}

\begin{IEEEbiography}[{\includegraphics[width=1in,height=1.25in,clip,keepaspectratio]{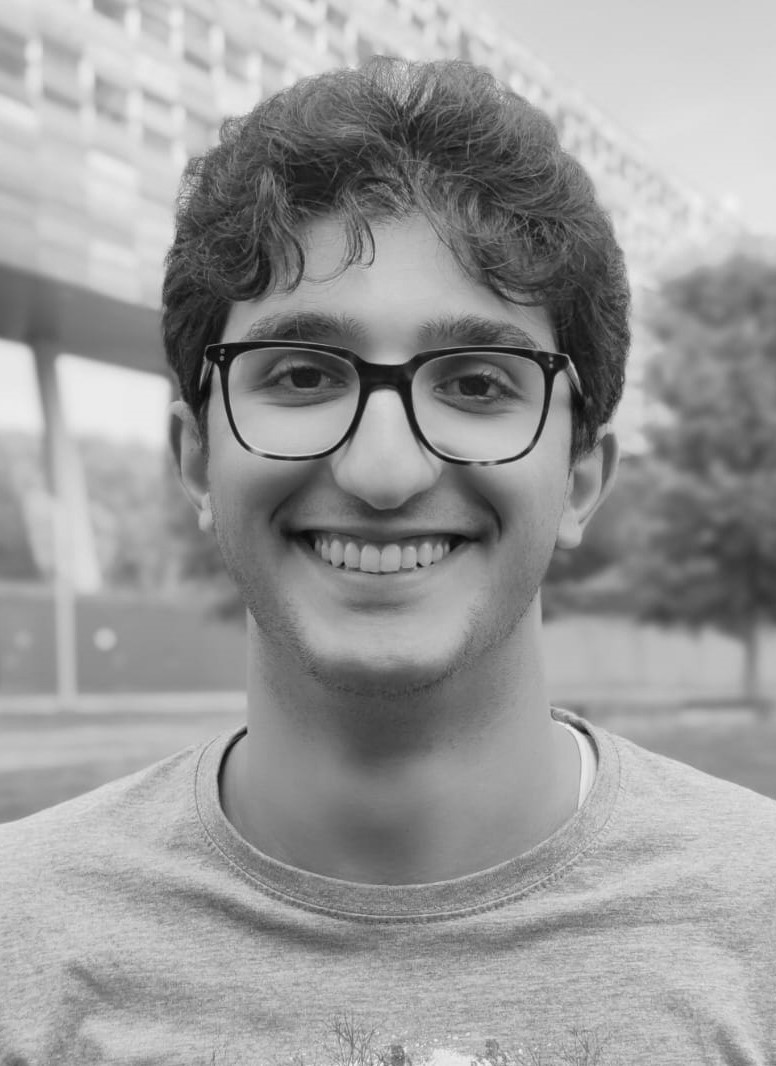}}]{Armin Pirastehzad} received the B.Sc. and
	M.Sc. degrees in electrical engineering from the
	University of Tehran, Tehran, Iran, in 2017 and
	2020, respectively. He is currently pursuing his Ph.D. studies in systems and control at the Bernoulli Institute for Mathematics, Computer Science and Artificial Intelligence, University of Groningen, Groningen, The Netherlands. His research interests include contract theory, modular system design, output regulation theory, and approximation
	techniques in control theory.
\end{IEEEbiography}
\vspace*{-1cm}
\begin{IEEEbiography}[{\includegraphics[width=1in,height=1.25in,clip,keepaspectratio]{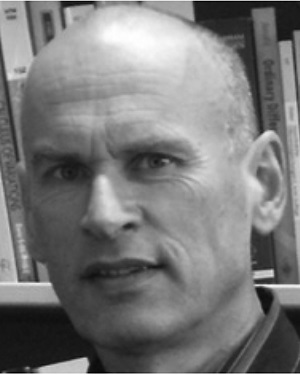}}]{Arjan van der Schaft} (Fellow, IEEE) received the Undergraduate and
	Ph.D. degrees in mathematics from the University of Groningen, Groningen,
	Netherlands, in 1979 and 1983, respectively.
	
	In 1982, he joined the Department of Applied Mathematics, University
	of Twente, Enschede, The Netherlands, where he was a Full Professor in
	Mathematical Systems and Control Theory in 2000. In September 2005,
	he joined the University of Groningen as a Professor in Mathematics. He
	has authored various books including Variational and Hamiltonian Control
	Systems (Springer, 1987, with P. E. Crouch), Nonlinear Dynamical Control
	Systems (Springer, 1990, 2016, with H. Nijmeijer), L2-Gain and Passivity
	Techniques in Nonlinear Control (Springer, 1996, 2000, 2017), An Introduction to Hybrid Dynamical Systems (Springer, 2000, with J.M. Schmacher),
	and Port-Hamiltonian Systems: An Introductory Overview (NOW Publisher,
	2014, with D. Jeltsema).
	
	Dr. Van der Schaft is a Fellow of the International Federation
	of Automatic Control (IFAC), and was the 2013 recipient of the three-yearly
	awarded Certificate of Excellent Achievements of the IFAC Technical Committee on Nonlinear Systems. He was an Invited Speaker at the International
	Congress of Mathematicians, Madrid, 2006.

\end{IEEEbiography}
\vspace*{-1cm}
\begin{IEEEbiography}[{\includegraphics[width=1in,height=1.25in,clip,keepaspectratio]{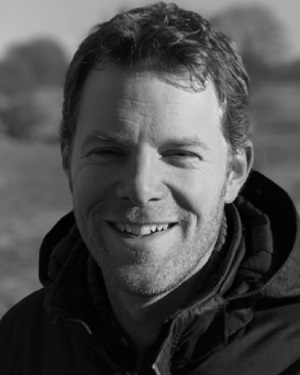}}]{Bart Besselink} (Member, IEEE) received the M.Sc. (cum laude) degree in mechanical engineering in 2008 and the Ph.D. degree in 2012, both from Eindhoven University of Technology, Eindhoven, The Netherlands.
	
Since 2016, he has been with the Bernoulli Institute for Mathematics, Computer Science and Artificial Intelligence, University of Groningen, Groningen, The Netherlands, where he is currently an associate professor. He was a short-term Visiting Researcher with the Tokyo Institute of Technology, Tokyo, Japan, in 2012. Between 2012 and 2016, he was a Postdoctoral Researcher with the ACCESS Linnaeus Centre and Department of Automatic Control, KTH Royal Institute of Technology, Stockholm, Sweden.
	
His main research interests are on mathematical systems theory for large-scale interconnected systems, with emphasis on contract-based design and control, compositional analysis, model reduction, and applications in intelligent transportation systems and neuromorphic computing. He is a recipient (with Xiaodong Cheng and Jacquelien Scherpen) of the 2020 Automatica Paper Prize.
\end{IEEEbiography}

\end{document}